\documentclass[preprint,12pt]{elsarticle}

\usepackage{amsmath,amssymb,mathtools,amsthm}
\usepackage{enumitem}
\usepackage{array,booktabs,tabularx}
\usepackage[table]{xcolor}
\usepackage{hyperref}

\theoremstyle{plain}
\newtheorem{theorem}{Theorem}[section]
\newtheorem{proposition}[theorem]{Proposition}
\newtheorem{lemma}[theorem]{Lemma}
\newtheorem{corollary}[theorem]{Corollary}
\theoremstyle{definition}
\newtheorem{definition}[theorem]{Definition}
\newtheorem{remark}[theorem]{Remark}
\newtheorem{example}[theorem]{Example}

\newcommand{\M}{\mathbf M}
\newcommand{\A}{\mathbf A}
\newcommand{\PosId}{\operatorname{Id}^{+}}
\newcommand{\Sk}{\operatorname{Sk}}
\newcommand{\Cr}{\mathcal C_{\mathrm r}}
\newcommand{\layer}[1]{L_{#1}}
\newcommand{\Comp}[1]{\mathbf L_{#1}}
\newcommand{\lay}[1]{\operatorname{L}_{#1}}

\newcommand{\leOmega}{\le_{\Omega}}

\newcommand{\Def}{\mathrel{:=}}

\title{Residual coherentization of balanced residuated partially ordered semigroups}

\author[aff1,aff2]{S\'andor Jenei\corref{cor1}}
\ead{jenei.sandor@uni-eszterhazy.hu}
\ead{jenei@ttk.pte.hu}
\cortext[cor1]{Corresponding author}

\affiliation[aff1]{
 organization={Institute of Mathematics and Informatics, Eszterh\'azy K\'aroly Catholic University},
 addressline={Le\'anyka utca 4},
 city={Eger},
 postcode={3300},
 country={Hungary}
}

\affiliation[aff2]{
 organization={Institute of Mathematics and Informatics, Faculty of Sciences, University of P\'ecs},
 addressline={Ifj\'us\'ag \'utja 6},
 city={P\'ecs},
 postcode={7624},
 country={Hungary}
}

\begin{document}

\begin{frontmatter}

\begin{abstract}
This paper develops a canonical decomposition--reconstruction theory for balanced residuated partially ordered semigroups.
The starting point is the intrinsic local-unit map
$
\tau(x)=x\backslash x=x/x,
$
whose values are positive idempotents.
The primitive fibres of this map are generally too fine to be compatible with multiplication and residuals: the local units of \(xy\), \(x\backslash y\), and \(x/y\) need not be determined by the local units of \(x\) and \(y\).
We therefore construct the residual coherentization \(\mathcal C_{\mathrm r}(\mathbf M)\), the finest quotient of the positive-idempotent skeleton on which these three local-unit outputs are well defined at quotient level.
The blocks of \(\mathcal C_{\mathrm r}(\mathbf M)\) define the canonical components, while the quotient skeleton records the target component for products and residuals.
Together with the component algebras, the product-shadow maps \(x\mapsto xq\), and the residual-shadow maps \(y\mapsto y/q\), these data reconstruct the original algebra.
The final part compares this construction with subsemilattice-steady visibility decompositions.
At every finite stage of the induced iterative decompositions, the partition obtained from residual coherentization is finer than the partition obtained from any subsemilattice-steady visibility choice.
Equivalently, each component produced by the subsemilattice-steady construction is a union of residual-coherent components. 
\end{abstract}


\end{frontmatter}

\section{Introduction}

Residuated partially ordered semigroups are ordered semigroups equipped with left and right residuals.
Thus multiplication is not only isotone, but also interacts with the order through two adjoint operations.
This places them at the common boundary of ordered semigroup theory, residuated algebra, lattice theory, and algebraic logic.

The lattice-ordered members of this landscape are the residuated lattices, introduced in the classical work of Ward and Dilworth \cite{WardDilworth1939}.
They now form one of the standard algebraic frameworks for substructural logics: in the same way in which Boolean algebras, Heyting algebras, and \(\mathrm{MV}\)-algebras provide algebraic semantics for classical, intuitionistic, and {\L}ukasiewicz logic, broad classes of residuated lattices serve as algebraic semantics for systems without some of the structural rules of classical logic \cite{GalatosJipsenKowalskiOno2007,CignoliDOttavianoMundici2000,Hajek1998}.
Prominent examples and subclasses include Boolean algebras, Heyting algebras, \(\mathrm{MV}\)-algebras, BL-algebras, MTL-algebras, G\"odel algebras, product algebras, hoops, cancellative and integral commutative residuated lattices, lattice-ordered groups, negative cones and intervals of lattice-ordered groups, and complete examples such as quantales, to name a few \cite{GalatosJipsenKowalskiOno2007,AglianoMontagna2003,EstevaGodo2001,Rosenthal1990}.

There is also a classical algebraic source coming from ideal theory.
If \(R\) is a unital ring, the lattice of two-sided ideals of \(R\), ordered by inclusion and equipped with ideal product, is a unital quantale; the two residuals are given by the corresponding left and right ideal quotients.
Thus ideal multiplication and ideal quotients provide a natural ring-theoretic model of residuated structure \cite{Rosenthal1990,Gilmer1972,LarsenMcCarthy1971}.
This connection is another reason why decomposition methods based on products, residuals, and local units are relevant beyond the algebraic semantics of substructural logics.

\smallskip
The point of the paper is not to study all these classes separately.
Rather, it isolates one structural mechanism present in many of them, and automatic in several of the most prominent commutative or monoidal cases: local units arising from self-residuals, and the canonical coarsening needed when primitive local-unit fibres are not compatible with products and residuals.
The present paper works with balanced residuated partially ordered semigroups, where the self-residuals of an element coincide and determine an intrinsic positive idempotent.
The equality of the two self-residuals is automatic in the commutative setting, since the two residuals then coincide.
The positivity part is automatic in the monoidal setting: if a unit is present, then it is below both \(x\backslash x\) and \(x/x\).
Thus the balanced semigroup setting isolates, without assuming either commutativity or a global unit, the local-unit phenomenon that is automatic in many familiar residuated contexts.
The aim is to use this local-unit information as the basis for a canonical component decomposition and reconstruction theorem.

\smallskip
Decomposition theory for residuated structures has several established sources.
One source is the ordinal-sum and component-sum analysis of chains in algebraic logic.
BL-chains, MTL-chains, hoops, and finite integral commutative residuated chains have been studied through Archimedean classes, cancellative components, ordinal-sum representations, split exact sequences, and related structural methods \cite{AglianoMontagna2003,Busaniche2005,MontagnaNogueraHorcik2006,Horcik2007,HorcikMontagna2009,Horcik2011,CastiglioniZuluaga2021}.
Another source is ordered semigroup theory, where classical decompositions use Green relations, Rees quotients, principal factors, semilattices of components, or ordinal-sum-like assemblies \cite{Green1951,Rees1940,Clifford1941,CliffordPreston1961,Howie1995}.
P\l{}onka sums give a universal-algebraic version of the same general pattern: components are indexed by a semilattice, and operations are evaluated after the arguments have been transported to the component selected by the semilattice operation \cite{Plonka1967}.
The present paper belongs to this broad component-and-transport tradition, but its indexing object is not chosen externally.
It is extracted from the local-unit geometry of the residuated semigroup itself.

The specific local-unit-based decomposition line used here began with \cite{JeneiGroupRepresentation}, first circulated as a preprint in 2019 and published in 2022.
There the local-unit map was introduced as a decomposition invariant for even and odd involutive commutative \(\mathrm{FL}_e\)-chains.
The universe was partitioned into fibres of this map, each fibre being indexed by a positive idempotent.
For the class studied there, those fibres carried layer algebras, and the product part of the reconstruction had the form of a P\l{}onka-sum reconstruction over the positive-idempotent skeleton, supplemented by a directed lexicographic reconstruction of the order.
This is the first occurrence, in the sense relevant for the present paper, of a systematic decomposition in which the local-unit map itself is used as the primary stratifying invariant: the algebra is partitioned into fibres of that map, the fibres are indexed by positive idempotents, and reconstruction is organized over the resulting positive-idempotent skeleton.

Subsequent work has developed related local-unit or idempotent-based stratifications in several residuated and ordered-algebraic settings.
Finite commutative idempotent involutive residuated lattices were analyzed in \cite{JipsenTuytValota2021}, finite involutive po-semilattices in \cite{JipsenSugimoto2022}, locally integral involutive po-semigroups in \cite{GilFerezJipsenSugimoto2023}, and locally integral involutive po-monoids and semirings in \cite{GilFerezJipsenLodhia2023}.
Balanced residuated partially ordered monoids and, more recently, balanced residuated partially ordered semigroups were studied in \cite{BGJPSm,BGJPS}.
Related announcements on balanced residuated posets and locally integral involutive residuated structures appear in \cite{Jipsen2024BalancedPosets,Sugimoto2024LocallyIntegral}.
These papers show that the local-unit viewpoint is not tied to one special class of chains.
Rather, it has become a recurring structural device in the study of residuated and involutive ordered algebraic systems.

The present paper works in the same ambient class as \cite{BGJPS}, namely balanced residuated partially ordered semigroups.
The results of \cite{BGJPS} are used here only as a comparison framework for steady and visibility-based decompositions, not as prerequisites for the results proved below.

\smallskip
Let
\[
 \M=\langle M,\le,\cdot,\backslash,/\rangle
\]
be a balanced residuated partially ordered semigroup.
For every element, the two self-residuals agree and are positive.
The common value is then automatically idempotent, so
\[
 x\backslash x=x/x\in\PosId(\M).
\]
We write
\[
 \tau(x)\Def x\backslash x=x/x
\]
and call \(\tau(x)\) the local unit of \(x\).
Thus \(\tau\) maps the algebra into the carrier of the positive-idempotent skeleton, namely into \(\PosId(\M)\).
The positive-idempotent skeleton itself is the ordered set introduced below.
The primitive local-unit layers are the fibres
\[
 \lay{u}:=\{x\in M:\tau(x)=u\},
 \qquad u\in\Sk(\M).
\]
A natural first question is whether these primitive layers are already the components of a reconstruction.
For this to hold, the layer of an output must be determined by the layers of its inputs.
In the present residuated setting this means that the local units of
\[
 xy,\qquad x\backslash y,\qquad x/y
\]
should be determined by the local units of \(x\) and \(y\).
This primitive-layer requirement is natural, but it is restrictive.
It asserts that the finest local-unit stratification is already compatible with multiplication and with both residuals.

The primitive full-skeleton P\l{}onka-sum framework was introduced for balanced residuated partially ordered monoids in \cite{BGJPSm}, while its semigroup-level extension, including chosen visible resolutions under suitable steadiness and fibrancy hypotheses, is developed in \cite{BGJPS}.
There the layers support a directed system of metamorphisms, and the original algebra is reconstructed from the corresponding P\l{}onka-type data.
An important feature of that residuated framework is that order recovery uses a paired transport into a common target component.
In the primitive \(\tau\)-layer setting, every component is indexed by a single positive idempotent.
Thus, for a given target component, there is a distinguished target idempotent.
One transport is given by multiplication with this idempotent, and the other is given by division by the same idempotent.
The two transported elements are then compared inside the target component.

The present paper keeps this paired-transport principle but changes the level at which the components are chosen.
Instead of assuming that the primitive or visible layers already provide the right reconstruction data, we pass to the canonical residual-coherent quotient of the positive-idempotent skeleton.
At this quotient level a target component may contain several positive idempotents, and no single one is distinguished as the transition idempotent.
The single paired transport is therefore replaced by two indexed families: the product-shadow maps \(x\mapsto xq\) and the residual-shadow maps \(y\mapsto y/q\), where \(q\) ranges over the positive idempotents of the target component.
Both families transport elements into the common target component.
Their ambient-order behavior is different: the product shadow \(xq\) lies above \(x\), while the residual shadow \(y/q\) lies below \(y\).
Thus, for a suitable target idempotent \(q\), the comparison
$
 xq\le y/q
$
inside the target component is a genuine ambient-order certificate.
This family-valued transport mechanism is used twice: multiplication is recovered from the product-shadow family, and order is recovered from paired families of product and residual shadows.

The contribution of the paper is fourfold.
First, it isolates the quotient-level coherence condition needed for products and both residuals, and constructs the canonical residual coherentization \(\Cr(\M)\).
Second, it shows that the resulting blocks carry balanced residuated component algebras and that the quotient skeleton determines the target component for all three operations.
Third, it proves a reconstruction theorem in which multiplication is recovered from the product-shadow family, order is recovered from paired families of product and residual shadows, and the residuals are then forced by residuation.
Finally, it compares the canonical decomposition induced by residual coherentization with subsemilattice-steady visibility decompositions.

\smallskip
We now describe the construction more explicitly.
Instead of assuming that the primitive \(\tau\)-layers already satisfy the product and residual local-unit laws, we canonically coarsen the positive-idempotent skeleton.
Whenever the current skeleton resolution does not determine the local unit of a product or residual output, the positive idempotents responsible for that failure are identified.
Iterating this procedure yields a canonical partition
\[
 \Cr(\M)
\]
of \(\Sk(\M)\), called the residual coherentization of \(\tau\).
It is the finest quotient of the positive-idempotent skeleton on which the local units of products and of both residuals are determined at the quotient level.
If \(A\) is a block of \(\Cr(\M)\), the corresponding component is
\[
 \layer{A}:=\{x\in M:\tau(x)\in A\}
          =\bigcup_{u\in A}\lay{u}.
\]
Thus the components are determined intrinsically by the algebra: they are not imposed by external choices or external data, and they need not be the primitive fibres of \(\tau\).
They are the components induced by the finest quotient of the skeleton on which the product and residual \(\tau\)-value laws are coherent.

The reconstruction theorem has three parts.
First, multiplication is recovered from the component multiplications and the product-shadow families into the target component.
More explicitly, if \(x\in\layer{A}\), \(y\in\layer{B}\), and \(C=A\vee B\) is the target block in the quotient skeleton, then
\[
 xy
 =
 \min\{(xr)(yr):r\in C\},
\]
where the minimum is taken in the target component \(\layer{C}\).
Thus the quotient skeleton identifies the component in which the product must lie, the product-shadow families move both factors into that component, and the internal multiplication of \(\Comp{C}\) computes the displayed candidates.
Second, order is recovered by a resolved paired-transport test.
The central order formula says that, for the same \(x\), \(y\), and \(C\),
\[
 x\le y
 \quad\Longleftrightarrow\quad
 \exists q\in C\quad xq\le_C y/q .
\]
The implication from right to left follows from positivity and residuation.
The converse uses residual coherence to place the relevant product-shadow and residual-shadow values in the same target component.
Third, once order and multiplication are recovered, the residuals are recovered by residuation.

The paper also compares residual coherentization with subsemilattice-steady visibility decompositions, in the sense of \cite{BGJPS}.
For the purposes of this comparison, the relevant hypotheses are recalled explicitly below.
Such a decomposition begins with a prior visibility choice, whereas \(\Cr(\M)\) is forced intrinsically by the product and residual \(\tau\)-value laws.
By contrast, \(\Cr(\M)\) starts from the full positive-idempotent skeleton and identifies only those distinctions which the product and residual \(\tau\)-value laws force one to identify.
The comparison theorem proves that this difference is structural.
At the first level, the residual coherentization refines every decomposition coming from a subsemilattice-steady visibility quotient.
Under iteration, the same refinement persists frontierwise.
Consequently, when both procedures terminate, every terminal subsemilattice-steady component is a union of terminal \(\tau\)-residuation-cohesive residual components.
In this refinement sense, residual coherentization is finer than every chosen-visibility quotient.

The paper is organized as follows.
Section~\ref{sec:balanced-notation} fixes the notation for balanced residuated partially ordered semigroups, positive idempotents, local units, and primitive layers.
Section~\ref{sec:motivation_narrow_scope} explains why primitive \(\tau\)-value laws are too narrow for the intended generality.
Section~\ref{sec:tau_residuation_coherence} defines \(\tau\)-residuation-coherent partitions and the residual coherentization \(\Cr(\M)\).
Section~\ref{sec:components} constructs the component algebras.
Section~\ref{sec:quotient-skeleton} records the quotient skeleton and its join operation.
Section~\ref{sec:shadow-maps} introduces the product-shadow and residual-shadow maps.
Sections~\ref{sec:paired-transport-order-recovery} and \ref{sec:algebra-reconstruction} prove order recovery and full algebra reconstruction.
Sections~\ref{sec:separating-comparison} and \ref{sec:Cr-refines-subsemilattice-steady} compare the construction with subsemilattice-steady visibility decompositions.

\section{Balanced residuated local-unit notation}
\label{sec:balanced-notation}

An element \(p\) of a partially ordered semigroup \(\langle M,\le,\cdot\rangle\) is called \emph{positive} if
\[
 z\le zp
 \qquad\text{and}\qquad
 z\le pz
 \qquad(z\in M).
\]
A \emph{positive idempotent} is a positive element \(p\) satisfying \(p^2=p\).

\begin{definition}[Balanced residuated partially ordered semigroup]
\label{def:balanced-residuated-posg}
A \emph{residuated partially ordered semigroup} is a structure
\[
\M=\langle M,\le,\cdot,\backslash,/\rangle
\]
such that \(\langle M,\le\rangle\) is a poset, \(\langle M,\cdot\rangle\) is a semigroup, and
\[
 xy\le z
 \quad\Longleftrightarrow\quad
 x\le z/y
 \quad\Longleftrightarrow\quad
 y\le x\backslash z.
\]
It is \emph{balanced} if, for every \(x\in M\),
\[
 x\backslash x=x/x
\]
and this common self-residual is positive.
It is then automatically idempotent: putting \(a=x\backslash x\), residuation
gives \(xa\le x\), while positivity gives \(x\le xa\), so \(xa=x\).
Hence \(xa^2=x\), and residuation yields \(a^2\le x\backslash x=a\), whereas
positivity gives \(a\le a^2\).
Thus \(a^2=a\).

For a balanced residuated partially ordered semigroup we write
\[
 \tau(x):=x\backslash x=x/x.
\]

\end{definition}

\noindent\textbf{Standing assumption.}
Throughout the paper, unless explicitly stated otherwise, \(\M\) denotes a balanced residuated partially ordered semigroup.
Thus \(\tau(x)=x\backslash x=x/x\) is always defined, and its values are positive idempotents.

Let
\[
   \PosId(\M)=\{u\in M:u^2=u\ \text{and}\ u\ \text{is positive}\}
\]
denote the set of positive idempotents of \(\M\).
The positive-idempotent skeleton is the ordered set
\[
   \Sk(\M)=\langle \PosId(\M),\leq\rangle,
\]
where the order is inherited from \(\M\).
As usual, we also write \(u\in\Sk(\M)\) as shorthand for \(u\in\PosId(\M)\).
The primitive \(\tau\)-layer at \(u\in\Sk(\M)\) is
\[
 \lay{u}:=\{x\in M:\tau(x)=u\}.
\]
For a set \(A\subseteq\Sk(\M)\), put
\[
 \layer{A}:=\{x\in M:\tau(x)\in A\}.
\]
If \(A\) is a block of a skeleton partition, \(\layer{A}\) is the corresponding component.

\begin{lemma}[Basic positive-idempotent facts]\label{lem:basic-positive-idempotent-facts}
Let \(u\in\Sk(\M)\).
Then, for every \(x\in M\), \(x\le xu\) and \(x/u\le x\).
If \(u\le\tau(x)\), then \(xu=x=ux\).
Moreover, every positive idempotent is central, and if \(u\) is central, then
\[
 x/u=u\backslash x .
\]
Finally, \(\tau(u)=u\), and the product of positive idempotents is their join in the positive-idempotent skeleton.
Thus, for \(u,v\in\Sk(\M)\),
\[
 uv=vu=u\vee v .
\]
\end{lemma}

\begin{proof}
The inequality \(x\le xu\) is positivity.
Also, \((x/u)u\le x\) by residuation, and positivity gives \(x/u\le (x/u)u\); hence \(x/u\le x\).

Since \(\tau(x)=x/x=x\backslash x\), residuation gives \(x\tau(x)\le x\) and \(\tau(x)x\le x\), while positivity gives the reverse inequalities.
Hence \(x\tau(x)=x=\tau(x)x\).
If \(u\le\tau(x)\), then isotonicity gives \(xu\le x\tau(x)=x\) and \(ux\le\tau(x)x=x\), while positivity gives \(x\le xu\) and \(x\le ux\).
Therefore \(xu=x=ux\).

We prove centrality.
Let \(u\in\Sk(\M)\).
Since \((xu)u=xu\), residuation gives \(u\le (xu)\backslash(xu)\).
By balancedness, \((xu)\backslash(xu)=(xu)/(xu)\), so \(u\le (xu)/(xu)\), and hence \(u(xu)\le xu\).
Positivity gives the reverse inequality, so \(uxu=xu\).
Similarly, from \(u(ux)=ux\) one obtains \(u\le (ux)/(ux)=(ux)\backslash(ux)\), and therefore \((ux)u\le ux\); positivity gives the reverse inequality, so \(uxu=ux\).
Hence \(ux=uxu=xu\), and \(u\) is central.

If \(u\) is central, then, for every \(z\in M\), one has \(z\le x/u\) iff \(zu\le x\), iff \(uz\le x\), iff \(z\le u\backslash x\).
Thus \(x/u=u\backslash x\).

For a positive idempotent \(u\), the equality \(uu=u\) gives \(u\le u/u\) and \(u\le u\backslash u\) by residuation.
Conversely, \((u/u)u\le u\) and \(u(u\backslash u)\le u\), while positivity gives \(u/u\le(u/u)u\) and \(u\backslash u\le u(u\backslash u)\).
Hence \(u/u=u=u\backslash u\), so \(\tau(u)=u\).

Finally, if \(u,v\in\Sk(\M)\), then \(uv=vu\) by centrality.
The element \(uv\) is a positive idempotent: positivity is inherited from \(u\) and \(v\), and idempotence follows from \((uv)^2=uvuv=u^2v^2=uv\).
Moreover \(u\le uv\) and \(v\le vu=uv\), so \(uv\) is an upper bound of \(u\) and \(v\).
If \(g\in\Sk(\M)\) is any upper bound of \(u\) and \(v\), then \(uv\le gv\le gg=g\).
Thus \(uv\) is the least upper bound of \(u\) and \(v\), that is, \(uv=vu=u\vee v\).
\end{proof}

\section{Motivation: from primitive exactness to quotient-level coherence}
\label{sec:motivation_narrow_scope}

The primitive \(\tau\)-layers are the most immediate local-unit resolution of a balanced residuated partially ordered semigroup.
They are, however, too rigid for a general decomposition theory.
At primitive resolution one asks that the local unit of each product and of each residual be determined by the local units of the two inputs.
Equivalently, for \(x\in\lay{u}\) and \(y\in\lay{v}\), the elements
\[
 xy,
 \qquad
 x\backslash y,
 \qquad
 x/y
\]
should all have local units prescribed by \(u\) and \(v\) alone.
Thus primitive exactness is a well-definedness requirement: multiplication and the two residuals must have target local units depending only on the two source local units.

This is a strong exactness requirement.
It is not a formal consequence of balancedness or of residuation; rather, it is an additional structural property of the chosen local-unit presentation.
Known positive instances of primitive or chosen-resolution exactness come from special structural settings: local-unit decompositions of even and odd involutive \(\mathrm{FL}_e\)-chains \cite{JeneiGroupRepresentation}, the primitive full-skeleton po-monoid construction of \cite{BGJPSm}, steady or fibrant local-unit presentations in the semigroup setting of \cite{BGJPS}, finite commutative idempotent involutive residuated lattices \cite{JipsenTuytValota2021}, and related locally integral involutive structures \cite{GilFerezJipsenSugimoto2023,GilFerezJipsenLodhia2023,JipsenSugimoto2022,Sugimoto2024LocallyIntegral}.
These are highly constrained situations; for instance, idempotency, involutivity, and local integrality each impose substantial algebraic restrictions, and hence strongly constrain the local-unit geometry.
Outside such special settings, verifying or refuting the product and residual \(\tau\)-value laws is a substantive problem.
Accordingly, the primitive \(\tau\)-fibres satisfy these laws only in special structural situations; in general balanced residuated partially ordered semigroups the laws can fail.

The construction below replaces primitive exactness by quotient-level exactness.
Instead of requiring the singleton \(\tau\)-fibres to satisfy product and residual local-unit laws, we allow the positive-idempotent skeleton to be coarsened.
The required coarsening is not chosen externally.
It is the finest quotient on which the local units of products and of both residuals are determined by the quotient-local units of the inputs.
This quotient is the residual coherentization \(\Cr(\M)\).
The conceptual shift is therefore
\[
 \text{primitive local-unit exactness}
 \quad\leadsto\quad
 \text{canonical quotient-level local-unit exactness}.
\]
The rest of the paper shows that this replacement loses no information: the quotient skeleton, the induced component algebras, and the product-shadow and residual-shadow maps reconstruct the original ordered residuated semigroup.

\section{Skeleton partitions and \texorpdfstring{\(\tau\)}{tau}-residuation coherence}
\label{sec:tau_residuation_coherence}

Let
\[
 S=\Sk(\M).
\]
By the standing assumption, the positive idempotents are central and their product is their join.
Thus \(S\) is a join-semilattice, and we write
\[
 u\vee v=uv
\]
for \(u,v\in S\).
A skeleton partition means a partition of \(S\).
If \(\Theta\) is such a partition, write \(u\equiv_\Theta v\) when \(u\) and \(v\) lie in the same \(\Theta\)-block, and write \([u]_\Theta\) for the \(\Theta\)-block of \(u\).

We order skeleton partitions by refinement:
\[
 \Theta\le \Phi
 \quad\Longleftrightarrow\quad
 u\equiv_\Theta v\Longrightarrow u\equiv_\Phi v
 \quad(u,v\in S).
\]
Thus larger partitions in this order are coarser.
The symbol \(\bigvee\) denotes the join in the lattice of partitions, equivalently the partition whose associated equivalence relation is generated by the union of the associated equivalence relations.
This convention is important below: the canonical coherent partition is the least coarsening, equivalently the finest partition, satisfying the required coherence condition.

For \(A\subseteq S\), put
\[
 \layer{A}:=\tau^{-1}(A)=\{x\in M:\tau(x)\in A\}.
\]
Thus \(\layer{A}\) is a union of primitive \(\tau\)-layers, not an independently chosen subset of \(M\).

\begin{definition}[\(\tau\)-saturated operation-output set]
\label{def:tau-saturated-operation-output-set}
For \(A,B\subseteq S\), define the \emph{\(\tau\)-saturated operation-output set} of the pair \((A,B)\) by
\[
 A\star_{\tau} B
 :=
 \{\tau(xy),\ \tau(x\backslash y),\ \tau(x/y):
      x\in\layer{A},\ y\in\layer{B}\}\subseteq S .
\]
\end{definition}

\begin{lemma}[Monotonicity of $\tau$-saturated operation-output sets]
\label{lem:tau-saturated-operation-output-monotone}
If \(A\subseteq A'\) and \(B\subseteq B'\), then
\[
 A\star_{\tau} B
 \subseteq
 A'\star_{\tau} B'.
\]
\end{lemma}

\begin{proof}
The inclusions \(\layer{A}\subseteq\layer{A'}\) and \(\layer{B}\subseteq\layer{B'}\) imply that every product or residual output obtained from \(\layer{A}\) and \(\layer{B}\) is also obtained from \(\layer{A'}\) and \(\layer{B'}\).
Applying \(\tau\) gives the inclusion.
\end{proof}

\begin{remark}
For every \(A\subseteq S\), one has
\[
 A\subseteq A\star_{\tau} A.
\]
Indeed, if \(u\in A\), then \(u\in\layer{A}\), and
\[
 u=\tau(uu)=\tau(u\backslash u)=\tau(u/u)
 \in A\star_{\tau} A.
\]
\end{remark}

\begin{definition}[Output blocks]
\label{def:output-set}
Let \(\Theta\) be a skeleton partition of \(S\), and let \(A,B\in \Theta\).
Define the set of \emph{output blocks} of the pair \((A,B)\) by
\[
 \operatorname{Out}_{\Theta}(A,B)
 :=
 \{C\in \Theta:
      C\cap (A\star_{\tau} B)\ne\varnothing\}.
\]
Since \(A\star_{\tau} B\ne\varnothing\), the set \(\operatorname{Out}_{\Theta}(A,B)\) is nonempty.
\end{definition}

\begin{definition}[\(\tau\)-multiplication and \(\tau\)-residuation coherence]
A skeleton partition \(\Theta\) is \emph{\(\tau\)-multiplication-coherent} if, for every pair of blocks \(A,B\in \Theta\), the set
\[
 \{\tau(xy):x\in\layer{A},\ y\in\layer{B}\}
\]
is contained in a single \(\Theta\)-block.

A skeleton partition \(\Theta\) is \emph{\(\tau\)-residuation-coherent} if
\[
 \operatorname{Out}_{\Theta}(A,B)
\]
is a singleton for every pair of blocks \(A,B\in \Theta\).
Equivalently, multiplication and the two residuals have a well-defined target component with respect to \(\Theta\): the \(\tau\)-component of \(xy\), \(x\backslash y\), and \(x/y\) is determined by the two input blocks alone.
If \(\Theta\) is \(\tau\)-residuation-coherent, we denote by
\[
 A\ast_\Theta B
\]
the unique block in \(\operatorname{Out}_{\Theta}(A,B)\).
\end{definition}

\begin{proposition}[Immediate consequences of \(\tau\)-residuation coherence]
\label{prop:immediate-consequences-tau-res-coherence}
Every \(\tau\)-residuation-coherent skeleton partition is \(\tau\)-multiplication-coherent.
Moreover, if \(\Theta\) is \(\tau\)-residuation-coherent and \(x\in\layer{A}\), \(y\in\layer{B}\), then
\[
 xy,
 \quad
 x\backslash y,
 \quad
 x/y
 \in \layer{A\ast_\Theta B}.
\]
For positive idempotents \(u\in A\) and \(v\in B\), one has
\[
 u\vee v=uv\in A\ast_\Theta B.
\]
\end{proposition}

\begin{proof}
The multiplication assertion follows because the product-output values are among the operation-output values.
The component-target assertion is exactly the definition of the target block \(A\ast_\Theta B\).
Finally, if \(u,v\in S\), then \(\tau(u)=u\), \(\tau(v)=v\), and
\[
 uv=u\vee v,
\]
because positive idempotents are central and their product is their join in the positive-idempotent skeleton.
Thus
\[
 u\vee v=\tau(uv)\in A\star_{\tau} B,
\]
and hence \(u\vee v\in A\ast_\Theta B\).
\end{proof}

\begin{proposition}[Induced quotient skeleton]
\label{prop:induced-quotient-skeleton}
Let \(\Theta\) be \(\tau\)-residuation-coherent.
Then the formula
\[
 A\vee B:=[u\vee v]_\Theta
 \qquad(u\in A,\ v\in B)
\]
is independent of the representatives \(u\in A\), \(v\in B\), and defines a join-semilattice structure on \(\Theta\).
Moreover,
\[
 A\vee B=A\ast_\Theta B.
\]
\end{proposition}

\begin{proof}
By Proposition~\ref{prop:immediate-consequences-tau-res-coherence}, every representative-wise join \(u\vee v=uv\), with \(u\in A\) and \(v\in B\), belongs to the block \(A\ast_\Theta B\).
Therefore all such joins lie in one \(\Theta\)-block, and the displayed formula is well-defined.
Commutativity, associativity, and idempotence are inherited from the join operation on \(S\), because the quotient operation has been shown to be representative-independent.
Thus \(\Theta\) is a join-semilattice, and its join is precisely the operation \((A,B)\mapsto A\ast_\Theta B\).
\end{proof}

\begin{definition}[Residual connected-component coarsening]
\label{def:residual-component-coarsening}
Let \(\Theta\) be a skeleton partition of \(S\).
Define a graph on the set of blocks of \(\Theta\) by joining two distinct blocks \(C,D\in \Theta\) whenever there exist blocks \(A,B\in \Theta\) such that
\[
 C,D\in \operatorname{Out}_{\Theta}(A,B).
\]
Equivalently, the edge condition says that both \(C\) and \(D\) meet the same saturated operation-output set
\[
 A\star_{\tau} B.
\]
We denote by
\[
 \operatorname{CC}_{\mathrm r}(\Theta)
\]
the skeleton partition obtained by merging precisely those blocks that lie in the same connected component of this graph.
\end{definition}

\begin{lemma}[Monotonicity of residual coarsening]
\label{lem:CCr-monotone}
If a skeleton partition \(\Theta\) refines a skeleton partition \(\Phi\), then
\[
 \operatorname{CC}_{\mathrm r}(\Theta)
 \le
 \operatorname{CC}_{\mathrm r}(\Phi).
\]
\end{lemma}

\begin{proof}
Let \(C,D\in \Theta\) be joined by an edge in the graph defining \(\operatorname{CC}_{\mathrm r}(\Theta)\).
Then there exist \(A,B\in \Theta\) such that
\[
 C,D\in\operatorname{Out}_{\Theta}(A,B).
\]
Let \(A^\Phi,B^\Phi,C^\Phi,D^\Phi\) be the \(\Phi\)-blocks containing \(A,B,C,D\), respectively.
By Lemma~\ref{lem:tau-saturated-operation-output-monotone},
\[
 A\star_{\tau} B
 \subseteq
 A^\Phi\star_{\tau} B^\Phi.
\]
Hence
\[
 C^\Phi,D^\Phi\in
 \operatorname{Out}_{\Phi}(A^\Phi,B^\Phi).
\]
If \(C^\Phi=D^\Phi\), then the edge collapses to one block of \(\Phi\).
If \(C^\Phi\ne D^\Phi\), then \(C^\Phi\) and \(D^\Phi\) are joined by an edge in the graph defining \(\operatorname{CC}_{\mathrm r}(\Phi)\).
Thus every edge for \(\Theta\) maps either to a single block or to an edge for \(\Phi\), and therefore connected components for \(\Theta\) map into connected components for \(\Phi\).
This proves the refinement assertion.
\end{proof}

\begin{proposition}[Fixed points of residual coarsening]
\label{prop:residual-component-coarsening-fixed-points}
For a skeleton partition \(\Theta\) of \(S\), the following are equivalent:
\begin{enumerate}
\item\label{item:res-coarsening-fixed}
\[
 \operatorname{CC}_{\mathrm r}(\Theta)=\Theta;
\]
\item\label{item:res-coarsening-coherent}
\(\Theta\) is \(\tau\)-residuation-coherent.
\end{enumerate}
\end{proposition}

\begin{proof}
Assume first that \(\operatorname{CC}_{\mathrm r}(\Theta)=\Theta\).
Let \(A,B\in \Theta\).
If
\[
 C,D\in\operatorname{Out}_{\Theta}(A,B),
\]
then \(C\) and \(D\) are joined by an edge in the graph defining \(\operatorname{CC}_{\mathrm r}(\Theta)\).
Since the partition is unchanged, no edge can join two distinct blocks.
Hence \(C=D\).
Thus \(\operatorname{Out}_{\Theta}(A,B)\) is a singleton, proving \(\tau\)-residuation coherence.

Conversely, assume that \(\Theta\) is \(\tau\)-residuation-coherent.
Then every output-block set \(\operatorname{Out}_{\Theta}(A,B)\) is a singleton.
Therefore no edge in the graph defining \(\operatorname{CC}_{\mathrm r}(\Theta)\) joins two distinct blocks.
Hence \(\operatorname{CC}_{\mathrm r}(\Theta)=\Theta\).
\end{proof}

\begin{definition}[Residual coherentization]
\label{con:canonical-tau-residuation-closure}
Let
\[
 \Theta_0:=\{\{u\}:u\in S\}
\]
be the singleton skeleton partition.
Define recursively
\[
 \Theta_{n+1}:=\operatorname{CC}_{\mathrm r}(\Theta_n)
 \qquad(n<\omega),
\]
and let \(\equiv_n\) be the equivalence relation on \(S\) whose classes are the blocks of \(\Theta_n\).
Since every step only merges blocks, the relations \(\equiv_n\) form an increasing sequence.
Set
\[
 u\equiv_{\omega}^{\mathrm r} v
 \quad\Longleftrightarrow\quad
 u\equiv_n v\ \text{for some }n<\omega .
\]
The union relation \(\equiv_{\omega}^{\mathrm r}\) is an equivalence relation, because an increasing union of equivalence relations is again an equivalence relation.
Define
\[
 \Cr(\M):=S/{\equiv_{\omega}^{\mathrm r}}.
\]
We call \(\Cr(\M)\) the  \emph{canonical \(\tau\)-residuation-coherent partition} of \(\M\), or equivalently the \emph{residual coherentization} of \(\M\).
The induced family
\[
 \mathcal L_{\mathrm r}(\M):=
 \{\layer{A}:A\in\Cr(\M)\}
\]
is the \emph{canonical \(\tau\)-residuation-coherent decomposition} of \(\M\), or equivalently the \emph{residual coherentization decomposition}.
If the iteration stabilizes at a finite stage, then this omega-limit is the fixed point at that stage.
The induced map
\[
 \tau_{\mathrm r}(x):=[\tau(x)]_{\Cr(\M)}
\]
will be called the \emph{canonical residual coherentization of \(\tau\)}.
\end{definition}

\begin{definition}[\(\tau\)-residuation-cohesive algebra]
\label{def:tau-residuation-cohesive}
A balanced residuated partially ordered semigroup \(\M\) is called \emph{\(\tau\)-residuation-cohesive} if its \(\tau\)-residuation-coherent partition is trivial, that is,
\[
 \Cr(\M)=\{S\}.
\]
Equivalently, its \(\tau\)-residuation-coherent decomposition is trivial, that is, it has a single component:
\[
 \mathcal L_{\mathrm r}(\M)=\{M\}.
\]
\end{definition}

\begin{proposition}[Cohesive algebras are terminal for residual coherentization]
\label{prop:tau-residuation-cohesive-terminal}
Let \(\A\) be a balanced residuated partially ordered semigroup.
One application of residual coherentization to \(\A\) produces no proper component refinement if and only if \(\A\) is \(\tau\)-residuation-cohesive.
Equivalently, in the recursive residual coherentization procedure, a component is terminal precisely when its induced component algebra is \(\tau\)-residuation-cohesive.
\end{proposition}

\begin{proof}
Applying residual coherentization to \(\A\) produces the component family induced by \(\Cr(\A)\).
This family has one component exactly when \(\Cr(\A)\) is the one-block partition of \(\Sk(\A)\), which is precisely the definition of \(\tau\)-residuation-cohesiveness for \(\A\).
\end{proof}

\begin{theorem}[Finest residual coherentization]
\label{prop:canonicality-Cr}
The partition \(\Cr(\M)\) is the finest \(\tau\)-residuation-coherent skeleton partition of \(S\).
Equivalently, every \(\tau\)-residuation-coherent skeleton partition \(\Phi\) satisfies
\[
 \Cr(\M)\le \Phi.
\]
By Proposition~\ref{prop:induced-quotient-skeleton}, the block set \(\Cr(\M)\) is a join-semilattice and is also \(\tau\)-multiplication-coherent.
\end{theorem}

\begin{proof}
We first prove that \(\Cr(\M)\) is \(\tau\)-residuation-coherent.
Let \(A,B\in\Cr(\M)\), and let
\[
 u,v\in A\star_{\tau} B.
\]
Choose
\[
 x,x'\in\layer{A},
 \qquad
 y,y'\in\layer{B}
\]
such that \(u\) and \(v\) are among the corresponding operation-output local units
\[
 \tau(xy),\ \tau(x\backslash y),\ \tau(x/y),
 \qquad
 \tau(x'y'),\ \tau(x'\backslash y'),\ \tau(x'/y').
\]
Since \(\tau(x),\tau(x')\in A\), there is \(n_A<\omega\) such that
\[
 \tau(x)\equiv_{n_A}\tau(x').
\]
Since \(\tau(y),\tau(y')\in B\), there is \(n_B<\omega\) such that
\[
 \tau(y)\equiv_{n_B}\tau(y').
\]
After replacing both stages by \(n:=\max\{n_A,n_B\}\), using monotonicity of the iteration, we have both
\[
 \tau(x)\equiv_n\tau(x'),
 \qquad
 \tau(y)\equiv_n\tau(y').
\]
Let \(A_n,B_n\in \Theta_n\) be the corresponding stage-\(n\) blocks.
Then
\[
 u,v\in A_n\star_{\tau} B_n.
\]
If \(U,V\in \Theta_n\) are the blocks containing \(u,v\), respectively, then
\[
 U,V\in\operatorname{Out}_{\Theta_n}(A_n,B_n).
\]
Hence \(U\) and \(V\) are merged by
\[
 \Theta_{n+1}=\operatorname{CC}_{\mathrm r}(\Theta_n),
\]
so \(u\equiv_{n+1}v\).
Thus \(A\star_{\tau} B\) is contained in a single block of \(\Cr(\M)\), proving \(\tau\)-residuation coherence.

Let \(\Phi\) be any \(\tau\)-residuation-coherent skeleton partition of \(S\).
By Proposition~\ref{prop:residual-component-coarsening-fixed-points},
\[
 \operatorname{CC}_{\mathrm r}(\Phi)=\Phi.
\]
The singleton partition \(\Theta_0\) refines \(\Phi\), so Lemma~\ref{lem:CCr-monotone} yields inductively that every \(\Theta_n\) refines \(\Phi\).
Hence each \(\equiv_{\omega}^{\mathrm r}\)-class is contained in a \(\Phi\)-block.
Therefore
\[
 \Cr(\M)\le\Phi.
\]
Thus \(\Cr(\M)\) refines every \(\tau\)-residuation-coherent skeleton partition, so it is the finest one.
The final assertion follows from Proposition~\ref{prop:induced-quotient-skeleton} and Proposition~\ref{prop:immediate-consequences-tau-res-coherence}.
\end{proof}

\begin{remark}[Relation with the primitive \(\tau\)-layers]
The primitive \(\tau\)-layers \(\lay{u}\) are not discarded.
They remain the layers of the original local-unit map \(x\mapsto\tau(x)\).
The canonical \(\tau\)-residuation-coherent components are obtained by iterated graph coarsening of the positive-idempotent skeleton.
The terminal components for this residual procedure are exactly the \(\tau\)-residuation-cohesive component algebras.
Equivalently, primitive layers are merged exactly when their positive idempotents become connected through a finite chain of operation-output obstructions.
Thus the coarsening is performed on the positive-idempotent skeleton, while the component partition of \(M\) is induced from it.
The quotient skeleton and its induced components are introduced in the following sections, after operation-output coherence has been obtained.
\end{remark}

\section{Components, quotient skeleton, shadow maps, and order recovery}
\subsection{Component algebras}
\label{sec:components}

Let \(\Omega\) be a \(\tau\)-residuation-coherent skeleton partition of \(\Sk(\M)\).
The canonical choice is \(\Omega=\Cr(\M)\).
For each block \(A\in \Omega\), define
\[
 \layer{A}:=\{x\in M:\tau(x)\in A\}.
\]
The \(\tau\)-residuation coherence of \(\Omega\) implies that \(\layer{A}\) is closed under multiplication and under both residuals.
Thus the \(\Omega\)-components are again balanced residuated partially ordered semigroups with the inherited operations.

\begin{definition}[Component algebra]
\label{def:component-algebra}
For \(A\in \Omega\), the \emph{component algebra at \(A\)} is
\[
 \Comp{A}:=\langle \layer{A},\le_A,\cdot_A,\backslash_A,/_A\rangle,
\]
where \(\le_A\), \(\cdot_A\), \(\backslash_A\), and \(/_A\) are the restrictions of the corresponding relation and operations of \(\M\) to the carrier \(\layer{A}\).
\end{definition}

\begin{proposition}[Component algebras]\label{prop:component-algebras}
Each \(\Comp{A}\) is a balanced residuated partially ordered semigroup.
Moreover,
\[
 \Sk(\Comp{A})=A=\Sk(\M)\cap\layer{A},
\]
and
\[
 A=\{\tau(x):x\in\layer{A}\}.
\]
The intrinsic local-unit map of \(\Comp{A}\) is the restriction of the ambient map \(\tau\).
\end{proposition}

\begin{proof}
Closure under multiplication and both residuals follows from \(\tau\)-residuation coherence applied to the pair \((A,A)\).
Indeed, the target block for \((A,A)\) is \(A\), since for every \(u\in A\) one has
\[
 u=uu=u\backslash u=u/u,
 \qquad \tau(u)=u.
\]
Thus multiplication and both residuals remain inside \(\layer{A}\).
Associativity, isotonicity, and the residuation laws are inherited from \(\M\), because the order and all three operations are restricted.
Balancedness is inherited in the same way.

It remains to identify the positive-idempotent skeleton of the component.
If \(u\in A\), then \(u\) is a positive idempotent of \(\M\), \(u\in\layer{A}\), and the ambient positivity inequalities restrict to \(\layer{A}\).
Thus \(u\in\Sk(\Comp{A})\).

Conversely, let \(p\in\Sk(\Comp{A})\).
Applying Lemma~\ref{lem:basic-positive-idempotent-facts} inside \(\Comp{A}\) gives \(\tau_{\Comp{A}}(p)=p\).
Since the component residuals are the restrictions of the ambient residuals,
\[
 \tau_{\Comp{A}}(p)=\tau(p).
\]
Hence \(\tau(p)=p\), so \(p\in\Sk(\M)\).
As \(p\in\layer{A}\), this gives \(p=\tau(p)\in A\).
Therefore
\[
 \Sk(\Comp{A})=A=\Sk(\M)\cap\layer{A}.
\]
It follows at once that
\[
 A=\{\tau(x):x\in\layer{A}\},
\]
and that the intrinsic local-unit map of \(\Comp{A}\) is the restriction of the ambient map \(\tau\).
\end{proof}

\begin{remark}
Unlike the primitive \(\tau\)-layer theory, a component \(\layer{A}\) need not be integrally closed and need not have a unique positive idempotent.
This is intentional.
The extra internal positive-idempotent structure records precisely what had to be merged in order to make the decomposition reconstructible without imposing steadiness axioms.
\end{remark}

For the canonical choice \(\Omega=\Cr(\M)\), the family
\[
 \mathcal L_{\mathrm r}(\M):=
 \{\Comp{A}:A\in \Cr(\M)\}
\]
will be called the \emph{canonical \(\tau\)-residuation-coherent decomposition} of \(\M\).

\subsection{The quotient skeleton}
\label{sec:quotient-skeleton}

The quotient skeleton associated with \(\Omega\) is the join-semilattice
\[
 \Sk_{\Omega}(\M):=\langle \Omega,\leOmega\rangle.
\]
Its elements are the positive-idempotent blocks.
The order is the join order:
\[
 A\leOmega B
 \quad\Longleftrightarrow\quad
 A\vee B=B.
\]
For \(x\in M\), write
\[
 \tau_\Omega(x):=[\tau(x)]_{\Omega}.
\]
When \(\Omega=\Cr(\M)\), this is the residual coherentization \(\tau_{\mathrm r}\) of \(\tau\).
Thus \(x\in\layer{A}\) exactly when \(\tau_\Omega(x)=A\).

\begin{proposition}[Target component]\label{prop:target-component}
Let \(x\in\layer{A}\) and \(y\in\layer{B}\).
Put
\[
 C=A\vee B.
\]
Then
\[
 xy,
 \quad
 x\backslash y,
 \quad
 x/y
 \in\layer{C}.
\]
Consequently multiplication and both residuals are component-targeted by the quotient skeleton.
\end{proposition}

\begin{proof}
This follows from \(\tau\)-residuation coherence of \(\Omega\) and Proposition~\ref{prop:induced-quotient-skeleton}.
\end{proof}

\subsection{Product and residual shadow maps}
\label{sec:shadow-maps}

The quotient skeleton gives the target component.
For comparable source and target components, the shadow maps give the component-to-component transports.

\begin{definition}[Product and residual shadow maps]
Let \(A,C\in\Sk_{\Omega}(\M)\) with \(A\leOmega C\), and let \(q\in C\).
For \(x\in\layer{A}\), define
\[
 \lambda^{A,C}_q(x):=xq,
 \qquad
 \eta^{A,C}_q(x):=x/q.
\]
When the source and target are clear, write simply
\[
 \lambda_q(x)=xq,
 \qquad
 \eta_q(x)=x/q.
\]
For each fixed \(q\in C\), the map \(\lambda_q\) is a product-shadow map from \(\layer{A}\) to \(\layer{C}\), and \(\eta_q\) is a residual-shadow map from \(\layer{A}\) to \(\layer{C}\).
The corresponding families are the product-shadow family and the residual-shadow family into \(C\).
\end{definition}

\begin{lemma}[Shadow maps land in the target component]\label{lem:shadow-maps-land-in-target-component}
If \(x\in\layer{A}\), \(A\leOmega C\), and \(q\in C\), then
\[
 xq\in\layer{C},
 \qquad
 x/q\in\layer{C}.
\]
\end{lemma}

\begin{proof}
Since \(q\in C\), \(\tau(q)=q\in C\).
The product-shadow map lands in \(\layer{C}\) by applying \(\tau\)-residuation coherence to \((x,q)\), whose target block is \(A\vee C=C\).
The residual-shadow map lands in \(\layer{C}\) by applying the right-residual trigger to the same pair.
Thus \(xq,x/q\in\layer{C}\).
\end{proof}

\begin{remark}[Why paired transport certifies order]
The transported element \(xq\) is an upper approximation of \(x\), because \(x\le xq\).
The transported element \(x/q\) is a lower approximation of \(x\), because \(x/q\le x\).
Thus a comparison
\[
 xq\le y/q
\]
inside a target component is a genuine certificate of the ambient comparison \(x\le y\).
This is the basic reason why the paired-transport comparison gives a direct order certificate rather than relying on a one-sided upper approximation.
\end{remark}

The preceding three constructions form a single package.
The quotient skeleton records the target block \(A\vee B\), the component algebras supply the internal operations and order inside each block, and the relevant shadow maps transport elements into the common target component.
The reconstruction theorem below shows that this package contains exactly the information needed to recover the ambient ordered residuated semigroup.

\subsection{Resolved paired-transport order recovery}
\label{sec:paired-transport-order-recovery}

Let \(\Omega\) be a \(\tau\)-residuation-coherent skeleton partition.
Then multiplication and both residuals land in target components determined by the quotient skeleton of \(\Omega\).
For the canonical reconstruction we later take \(\Omega=\Cr(\M)\), but the order-recovery argument itself works for every such \(\Omega\).
The residuals need not be recovered by an independent maximization formula based on transported elements: after the product and the ambient order are recovered, the residuals are forced by residuation.
The remaining point is therefore the ambient order, and it is recovered by the resolved paired-transport test below.

Let \(x\in\layer{A}\) and \(y\in\layer{B}\).
Put \(C=A\vee B\) in the quotient skeleton induced by \(\Omega\).
Define
\[
 x\preceq_C y
 \quad\Longleftrightarrow\quad
 \exists q\in C\quad xq\le_C y/q.
\]

\begin{lemma}[Soundness of resolved order witnesses]\label{lem:soundness-resolved-order-witnesses}
Let \(x\in\layer{A}\), \(y\in\layer{B}\), and \(C=A\vee B\).
If \(x\preceq_C y\), then \(x\le y\) in \(\M\).
\end{lemma}

\begin{proof}
Choose \(q\in C\) such that \(xq\le_C y/q\).
Since \(q\) is positive, \(x\le xq\).
By residuation, \((y/q)q\le y\); since \(q\) is positive, \(y/q\le (y/q)q\), and hence \(y/q\le y\).
Thus
\[
 x\le xq\le y/q\le y.
\]
\end{proof}

\begin{theorem}[Resolved order-completeness]\label{thm:resolved-order-completeness}
Every \(\tau\)-residuation-coherent skeleton partition satisfies the resolved order formula.
More explicitly, for all \(x\in\layer{A}\) and \(y\in\layer{B}\), with \(C=A\vee B\),
\[
 x\le y
 \quad\Longleftrightarrow\quad
 \exists q\in C\quad xq\le_C y/q.
\]
Indeed, whenever \(x\le y\), the single witness
\[
 q=\tau(x)\tau(y)=\tau(x)\vee\tau(y)
\]
works.
Equivalently, since this witness is central, the same target comparison may be written as
\[
 qx=xq\le_C q\backslash y.
\]
\end{theorem}

\begin{proof}
The right-to-left implication is Lemma~\ref{lem:soundness-resolved-order-witnesses}.
For the converse, assume \(x\le y\), and put
\[
 u=\tau(x),
 \qquad
 v=\tau(y),
 \qquad
 w=uv=u\vee v.
\]
Then \(u\in A\), \(v\in B\), and Proposition~\ref{prop:induced-quotient-skeleton} gives \(w\in C=A\vee B\).
The element \(w\) is a central positive idempotent.
Also \(u\) and \(v\) are two-sided local units for \(x\) and \(y\), respectively, so
\[
 ux=xu=x,
 \qquad
 vy=yv=y.
\]
We first record the elementary comparison equivalence induced by \(w\).
If \(x\le y\), then
\[
 wx=uvx=vux=vx\le vy=y.
\]
Conversely, if \(wx\le y\), then positivity of \(w\) gives \(x\le wx\le y\).
Hence
\[
 x\le y
 \quad\Longleftrightarrow\quad
 wx\le y.
\]
Since \(w^2=w\) and \(w\) is central,
\[
 wx=wwx=wxw.
\]
Therefore
\[
 x\le y
 \quad\Longleftrightarrow\quad
 wx\le y
 \quad\Longleftrightarrow\quad
 wwx\le y
 \quad\Longleftrightarrow\quad
 wxw\le y.
\]
By residuation, the last inequality is equivalent to
\[
 wx\le y/w.
\]
Since \(w\) is central, \(y/w=w\backslash y\), and since \(wx=xw\), the same condition may also be written as
\[
 xw=wx\le w\backslash y.
\]
Lemma~\ref{lem:shadow-maps-land-in-target-component} gives \(xw,y/w\in\layer{C}\), so this is a comparison inside the target component:
\[
 xw\le_C y/w.
\]
Thus \(x\preceq_C y\), with the single witness \(q=w\).
\end{proof}

\begin{corollary}[Order recovery for the residual coherentization]
For \(\Omega=\Cr(\M)\), and for \(x\in\layer{A}\), \(y\in\layer{B}\), \(C=A\vee B\),
\[
 x\le y
 \quad\Longleftrightarrow\quad
 \exists q\in C\quad xq\le_C y/q.
\]
\end{corollary}

\begin{remark}[Why the paired-transport test is sound]
A one-sided upper approximation of the form
\[
 y<\lambda_q(x)
\]
inside a target component is not, by itself, an ambient-order certificate.
The resolved test instead uses a target idempotent \(q\) and the comparison
\[
 \lambda_q(x)\le \eta_q(y),
\]
equivalently \(xq\le y/q\), inside the common target component.
This comparison is automatically sound: since \(q\) is positive, \(x\le xq\), and by residuation \(y/q\le y\).
Thus \(xq\le y/q\) implies \(x\le y\).
The witness \(q=\tau(x)\tau(y)\) supplied in Theorem~\ref{thm:resolved-order-completeness} then gives completeness of the paired-transport test.
\end{remark}

\begin{proposition}[Residuals are determined by order and product]\label{prop:residuals-determined-by-order-product}
Let \(\M=\langle M,\le,\cdot,\backslash,/\rangle\) be a residuated partially ordered semigroup.
Then the residuals are uniquely determined by \(M\), \(\le\), and \(\cdot\): for all \(x,y\in M\),
\[
 x\backslash y=\max\{z\in M:xz\le y\},
 \qquad
 x/y=\max\{z\in M:zy\le x\}.
\]
\end{proposition}

\begin{proof}
By residuation,
\[
 z\le x\backslash y
 \quad\Longleftrightarrow\quad
 xz\le y.
\]
Hence \(x\backslash y\) is exactly the greatest element of \(\{z:xz\le y\}\).
Similarly,
\[
 z\le x/y
 \quad\Longleftrightarrow\quad
 zy\le x,
\]
so \(x/y\) is exactly the greatest element of \(\{z:zy\le x\}\).
Thus two residuated structures on the same ordered semigroup have the same residuals.
\end{proof}

\section{Algebra reconstruction from the transport data}
\label{sec:algebra-reconstruction}

Fix
\[
 \Omega:=\Cr(\M).
\]
The reconstruction data are
\[
\begin{aligned}
\mathfrak D_{\mathrm r}(\M):=
\bigl(&\Sk_{\mathrm r}(\M),
 \{\Comp{A}:A\in\Cr(\M)\},\\
&\{\lambda^{A,C}_q:A\leOmega C,\ q\in C\},
 \{\eta^{A,C}_q:A\leOmega C,\ q\in C\}\bigr).
\end{aligned}
\]
where
\[
 \Sk_{\mathrm r}(\M)=\Sk_{\Cr(\M)}(\M),
\]
and \(\leOmega\) denotes the order of this residual quotient skeleton, and, for \(x\in \layer{A}\),
\[
 \lambda^{A,C}_q(x)=xq,
 \qquad
 \eta^{A,C}_q(x)=x/q .
\]
When the source and target components are clear, we write simply \(\lambda_q(x)\) and \(\eta_q(x)\).

\subsection{Product reconstruction}

Let \(x\in\layer{A}\) and \(y\in\layer{B}\), and put \(C=A\vee B\).
The resolved product candidates are
\[
 \Pi_C(x,y):=
 \{(xr)(yr):r\in C\}\subseteq\layer{C}.
\]

\begin{proposition}[Product recovery]\label{prop:product-recovery}
The product \(xy\) is the least element of \(\Pi_C(x,y)\):
\[
 xy=\min \Pi_C(x,y).
\]
\end{proposition}

\begin{proof}
For every \(r\in C\), positivity gives \(xy\le xyr\).
Since positive idempotents are central,
\[
 (xr)(yr)=xyr^2=xyr.
\]
Thus \(xy\le (xr)(yr)\) for every \(r\in C\).
On the other hand, \(\tau(xy)\in C\) by \(\tau\)-residuation coherence, and
\[
 (x\tau(xy))(y\tau(xy))=xy\tau(xy)=xy.
\]
Hence \(xy\) belongs to \(\Pi_C(x,y)\), and it is its least element.
\end{proof}

\subsection{Order reconstruction}

\begin{theorem}[Order recovery from paired transport]\label{thm:order-recovery-from-paired-transport}
For \(x\in\layer{A}\), \(y\in\layer{B}\), and \(C=A\vee B\),
\[
 x\le y
 \quad\Longleftrightarrow\quad
 \exists q\in C\quad xq\le_C y/q.
\]
\end{theorem}

\begin{proof}
This is Theorem~\ref{thm:resolved-order-completeness} applied to \(\Omega=\Cr(\M)\).
\end{proof}

\subsection{Residual reconstruction}

\begin{proposition}[Residual recovery from recovered order and product]\label{prop:residual-recovery-from-recovered-order-product}
Assume that the product and order of \(\M\) have been recovered from \(\mathfrak D_{\mathrm r}(\M)\).
Then the two residuals are recovered uniquely by
\[
 x\backslash y=\max\{z\in M:xz\le y\},
 \qquad
 x/y=\max\{z\in M:zy\le x\}.
\]
\end{proposition}

\begin{proof}
This is Proposition~\ref{prop:residuals-determined-by-order-product} applied to the recovered ordered semigroup.
Since the original structure is residuated, the indicated greatest elements exist and are the original residual values.
\end{proof}

\subsection{Full reconstruction}

\begin{theorem}[Resolved \(\tau\)-residuation-coherent decomposition--reconstruction]\label{thm:resolved-tau-residuation-reconstruction}
Let \(\M\) be a balanced residuated partially ordered semigroup, and let \(\Omega=\Cr(\M)\).
Then \(\M\) is recovered, up to the evident isomorphism, from the following data:
\begin{enumerate}[label=\textup{(\arabic*)},leftmargin=2.2em]
\item the quotient skeleton \(\Sk_{\mathrm r}(\M)\);
\item the component algebras \(\Comp{A}\), for \(A\in\Cr(\M)\);
\item the product-shadow maps \(\lambda_q(x)=xq\), where \(q\) ranges over the target block;
\item the residual-shadow maps \(\eta_q(x)=x/q\), where \(q\) ranges over the target block.
\end{enumerate}
The product is reconstructed as the least resolved product value.
The order is reconstructed by the resolved paired-transport formula.
The residuals are then reconstructed, on their own, as the residuals determined by this recovered ordered semigroup, as in Proposition~\ref{prop:residuals-determined-by-order-product}.
\end{theorem}

\begin{proof}
The universe is the disjoint union of the component universes.
The quotient skeleton determines the target component \(C=A\vee B\) for every pair of source components.
Product recovery gives \(xy\) inside \(\layer{C}\).
Theorem~\ref{thm:order-recovery-from-paired-transport} gives the ambient order.
Once the order and the product have been recovered, Proposition~\ref{prop:residuals-determined-by-order-product} recovers both residuals.
All reconstructed operations and the reconstructed order therefore agree with the original ones.
\end{proof}

\section{A separating comparison with primitive-layer and subsemilattice-steady visibility decompositions}
\label{sec:separating-comparison}

The preceding sections establish the canonical residual coherentization and the reconstruction data carried by it.

We now compare this construction with two stages of the balanced-residuated P\l{}onka-sum program: the primitive full-skeleton po-monoid construction of \cite{BGJPSm} and the steady, visibility-based semigroup construction of \cite{BGJPS}.
In both cases the ambient language is residuated, so the comparison is a genuine comparison of decomposition principles.
The point of comparison is structural: primitive exactness, steadiness, or visibility obtains coherence by imposing suitable local-unit exactness at a prescribed resolution, whereas the present method obtains coherence canonically, by forming the finest quotient of the full positive-idempotent skeleton forced by the product and residual \(\tau\)-value laws.

At primitive resolution, no coarsening is being performed.
In that case the present definitions reduce to the primitive \(\tau\)-value laws for the original \(\tau\)-layers: the local units of products and residuals must already be prescribed by the primitive positive-idempotent skeleton.
Equivalently, the primitive partition is already \(\tau\)-residuation-coherent exactly when the relevant primitive product and residual \(\tau\)-value conditions hold, namely the product condition together with both residual orientations.
At monoid level, this is precisely the product-and-residual local-unit exactness underlying the construction of \cite{BGJPSm}.
In the semigroup formulation of \cite{BGJPS}, the two residual orientations are equivalent under the hypotheses considered there, yielding condition \((H)\).
Thus steadiness or fibrancy imposes exactness at a chosen primitive or visible resolution, whereas the present construction obtains the required exactness by canonical coarsening.

The slogan is therefore
\[
 \text{steady exactness at a chosen resolution}
 \quad\leadsto\quad
 \text{residual coherence by canonical coarsening}.
\]

We now separate the remaining comparison points in turn.
The first subsection constructs a product example.
The following subsections interpret it against primitive full-skeleton steadiness and chosen subsemilattice visibility.

\subsection{The product example}

The example lies in the familiar class of finite commutative Boolean quantales, equivalently finite complete atomic Boolean residuated lattices with commutative multiplication.
It is built from complex algebras of two elementary commutative monoids.

\begin{example}[Residual coarsening and subsemilattice visibility]
\label{ex:C2_times_max3_residual_visibility}
Let \(C_2=\{0,a\}\) be the two-element group, written additively, and let
\(\mathbf B=\mathcal P(C_2)\) be its complex residuated lattice, that is, the
powerset algebra equipped with setwise multiplication and the corresponding
residuals; see \cite{GalatosJipsenKowalskiOno2007}.
Put \(u=\{0\}\) and \(g=C_2\).
Then \(\tau_{\mathbf B}(X)=u\) for the two singleton subsets of \(C_2\), while
\(\tau_{\mathbf B}(X)=g\) for \(X=\emptyset\) and \(X=g\).
A direct check gives
\[
 \tau_{\mathbf B}(X+Y)=\tau_{\mathbf B}(X)\tau_{\mathbf B}(Y),
 \qquad
 \tau_{\mathbf B}(X\backslash Y)=\tau_{\mathbf B}(X)\tau_{\mathbf B}(Y),
\]
and, by commutativity, the same holds for the right residual.

Let \(T=\{0,1,2\}\) with multiplication \(rs=\max\{r,s\}\), and let \(\mathbf D=\mathcal P(T)\) be its complex residuated lattice.
The positive idempotents of \(\mathbf D\) are
\[
 d_0=\{0\},\qquad h_1=\{0,1\},\qquad h_2=\{0,2\},\qquad k=T.
\]
The primitive \(\tau\)-partition of \(\mathbf D\) is \(\tau\)-multiplication-coherent, but it is not \(\tau\)-residuation-coherent.
Indeed,
\[
 \tau_{\mathbf D}(\{1\})=h_1,
 \qquad
 \tau_{\mathbf D}(\{0\})=d_0,
\]
whereas
\[
 \{1\}\backslash\{0\}=\emptyset,
 \qquad
 \tau_{\mathbf D}(\{1\}\backslash\{0\})=k.
\]
Thus residual coherence forces \(h_1\sim k\).
Similarly, \(\{0,2\}\backslash\{0\}=\emptyset\) forces \(h_2\sim k\).
The finite verification is transparent from the \(\tau\)-values: \(\tau_{\mathbf D}(Z)=d_0\) holds only for \(Z=\{0\}\), the three operations send the pair \(\{0\},\{0\}\) back to \(\{0\}\), and every other input pair has product and residual outputs whose \(\tau\)-value is one of \(h_1,h_2,k\).
Thus no witness forces \(d_0\) into the block containing \(h_1,h_2,k\).
Hence
\[
 \Cr(\mathbf D)=\{\{d_0\},\{h_1,h_2,k\}\}.
\]

Now put
\[
 \mathbf A=\mathbf B\times\mathbf D.
\]
Then \(\tau_{\mathbf A}\) is computed coordinatewise, and the primitive positive-idempotent skeleton is
\[
 \{u,g\}\times\{d_0,h_1,h_2,k\}.
\]
The first coordinate supplies one independent direction, while the second coordinate supplies a genuine residual-coherence obstruction.
Since both factors have primitive multiplication coherence, the primitive partition of \(\mathbf A\) is multiplication-coherent.
However, that primitive partition is not \(\tau\)-residuation-coherent: for
\[
 x=(\{0\},\{1\}),
 \qquad
 y=(\{0\},\{0\}),
\]
the product of the input local units is
\[
 (u,h_1)(u,d_0)=(u,h_1),
\]
whereas
\[
 \tau_{\mathbf A}(x\backslash y)=(u,k).
\]
Thus \((u,h_1)\) must be identified with \((u,k)\), and the analogous witness using \(\{0,2\}\) identifies \((u,h_2)\) with \((u,k)\).
The same witnesses with first-coordinate input of \(\tau_{\mathbf B}\)-value \(g\) force
\[
 (g,h_1)\sim(g,k),
 \qquad
 (g,h_2)\sim(g,k).
\]
No witness involves \((u,d_0)\) or \((g,d_0)\) in these forced residual collapses.
Consequently
\[
 \Cr(\mathbf A)=
 \bigl\{\{(u,d_0)\},\{(g,d_0)\}\bigr\}
 \cup
 \bigl\{\{(u,h_1),(u,h_2),(u,k)\},\{(g,h_1),(g,h_2),(g,k)\}\bigr\}.
\]
Thus \(\Cr(\mathbf A)\) has four blocks.
It is strictly coarser than the primitive partition, so residual coarsening is genuinely needed.
At the same time, \(\Cr(\mathbf A)\) still has a non-linear quotient skeleton: the quotient is the four-element diamond obtained from the two chains
\[
 u<g,
 \qquad
 \{d_0\}<\{h_1,h_2,k\}.
\]
Thus the residual coherentization performs exactly the necessary residual collapse in the second coordinate while preserving the independent diamond geometry.
\end{example}

\subsection{Comparison with primitive and subsemilattice-steady visibility}

The example separates the present construction from the primitive full-skeleton version of the balanced-residuated P\l{}onka-sum decomposition scheme introduced for po-monoids in \cite{BGJPSm}.
In that theory, at primitive resolution, the fibres must already be closed not only under multiplication but also under the residuals: the local units
$
 \tau(x\backslash y)
 $ and $
 \tau(x/y)
$
must already be determined by the local units \(\tau(x)\) and \(\tau(y)\).
The displayed witness violates exactly this requirement.
Here
\[
 \tau(x)=(u,h_1),
 \qquad
 \tau(y)=(u,d_0).
\]
At primitive resolution, the residual \(\tau\)-value law would require
\[
 \tau(x\backslash y)
 =
 \tau(x)\tau(y)
 =
 (u,h_1)(u,d_0)
 =
 (u,h_1).
\]
But the actual residual output satisfies
\[
 \tau(x\backslash y)=(u,k).
\]
Thus the primitive full-skeleton construction cannot use these fibres as its component system.
One must either choose a coarser visible subsemilattice, as in the subsemilattice-steady method of~\cite{BGJPS}, or pass canonically to the residual coherentization \(\Cr(\mathbf A)\).

For the subsemilattice-steady method, the natural coarse visible subsemilattice is
\[
 I=\{u,g\}\times\{d_0\}.
\]
This visibility forgets the nontrivial \(D\)-coordinate at the root.
For \(z=(X,Z)\in A\), the associated \(I\)-local unit is
\[
 u_z=(\tau_{\mathbf B}(X),d_0).
\]
Thus the second coordinate, where the residual obstruction lives, is invisible to the chosen subsemilattice.
The steadiness equations over \(I\) hold because the first coordinate satisfies the product and residual \(\tau\)-value laws, while the second coordinate is constantly projected to \(d_0\).
This gives the two visible fibres
\[
 A_{(p,d_0)}=B_p\times\mathcal P(T),
 \qquad p\in\{u,g\}.
\]
Equivalently, the visibility partition induced by \(I\) is
\[
 \mathcal P_I
 =
 \bigl\{\{u\}\times\{d_0,h_1,h_2,k\},
        \{g\}\times\{d_0,h_1,h_2,k\}\bigr\}.
\]
Hence the subsemilattice presentation treats the example at a coarser chosen visibility: it yields only the two visible fibres over \((u,d_0)\) and \((g,d_0)\), whereas \(\Cr(\mathbf A)\) has the four components displayed above.
The point is precisely that, relative to this chosen \(I\)-presentation, the four-component split is not merely postponed to a later visible stage.
The visible local-unit map has already replaced every full local unit \((p,r)\in\{u,g\}\times\{d_0,h_1,h_2,k\}\) by its projection
\[
 (p,r)\longmapsto (p,d_0).
\]
Thus the idempotents
\[
 d_0,\qquad h_1,\qquad h_2,\qquad k
\]
of the \(\mathbf D\)-coordinate are not separate visible skeleton points in this presentation.
They occur only through their common projection to \(d_0\).
Consequently the residual obstruction carried by these four idempotents is not available as a further visible skeleton inside the same chosen subsemilattice-steady presentation.
In this fixed presentation, the missing \(\mathbf D\)-split is therefore not postponed; it is suppressed by the initial visibility choice.

The preceding argument concerns the fixed visible subsemilattice \(I=\{u,g\}\times\{d_0\}\).
To rule out the possibility that the missing \(\mathbf D\)-split is recovered by choosing a different nontrivial visible skeleton inside the obstructing coordinate, the next subsection isolates the \(\mathbf D\)-factor itself.
It shows that \(\mathbf D\) admits no nontrivial subsemilattice-steady presentation which could account for the forced residual split.

\begin{remark}[Chosen visibility versus canonical forced coarsening]
\label{rem:chosen-visibility-versus-forced-coarsening}
The preceding example should not be read as saying that the algebra lies outside the subsemilattice-steady framework of~\cite{BGJPS}.
Rather, that framework treats the same algebra at a coarser chosen visibility.
The obstructing \(D\)-coordinate is not eliminated by the subsemilattice-steady presentation; it is made invisible at the root by projecting the full local unit onto a smaller visible subsemilattice.

This explains the structural difference between the two constructions.
A subsemilattice-steady presentation first chooses a visible subsemilattice
$
 I\subseteq\Sk(\mathbf A)
$
and replaces the full local unit \(\tau(a)\) by its projection
$
 u_a=\max\{p\in I:pa=a\}
     =\max\{p\in I:ap=a\}.
$
Thus a quotient of the full skeleton is built into the input data.
The product and residual local-unit laws are then required only after this projection.

Residual coherentization proceeds in the opposite order.
It starts from the full positive-idempotent skeleton and quotients canonically only those distinctions which are forced away by the product and residual \(\tau\)-value laws.
Thus the contrast is chosen visibility versus canonical forced coarsening: the subsemilattice-steady method obtains coherence by selecting a coarser visible skeleton, whereas \(\Cr(\mathbf A)\) keeps the full skeleton visible and then coarsens only where the residual \(\tau\)-value laws require it.

The reason residual coherentization refines every subsemilattice-steady visibility quotient is that such a presentation builds a coherent quotient of the full positive-idempotent skeleton into the input data.
The next section shows that the visibility partition induced by a subsemilattice-steady presentation is already \(\tau\)-residuation-coherent, and hence, by canonicality, is coarser than \(\Cr(\mathbf A)\).
The same restriction argument then propagates this comparison frontierwise through the corresponding decomposition trees.

The product example explains how a chosen subsemilattice visibility can hide the obstructing coordinate at the root.
It also indicates why the comparison is not a mere width--height trade-off: the missing split is not produced by continuing the same chosen-visibility presentation.
The next example removes this possible ambiguity completely by isolating the \(\mathbf D\)-coordinate and showing that it admits no nontrivial subsemilattice-steady presentation.
\end{remark}

\subsection{Beyond nontrivial chosen visibility}

The preceding product example shows that residual coherentization can preserve an independent diamond in the quotient skeleton while separating canonical forced coarsening from chosen subsemilattice visibility.
It also shows that the subsemilattice-steady method may obtain steadiness by hiding the obstructing coordinate at the root.
However, it does not rule out a width--height trade-off: the coarser first-level fibres still contain the obstructing \(\mathbf D\)-coordinate and may be refined further by a more discriminating residuated analysis.
The next example removes this possible interpretation.
It isolates the obstructing max-semilattice complex algebra itself.
For this algebra no nontrivial subsemilattice of the positive-idempotent skeleton gives a steady presentation: every nontrivial choice of visible skeleton still sees a residual \(\tau\)-value failure.
Nevertheless the residual coherentization is nontrivial and canonical.
Thus the example separates chosen coarse visibility from canonical coarsening in a sharper form: the subsemilattice method has only the trivial presentation, whereas \(\Cr\) still produces a proper two-block decomposition.

\begin{example}[Residual coarsening beyond all nontrivial subsemilattice-steady presentations]
\label{ex:max3_no_nontrivial_subsemilattice_steady}
Let \(T=\{0,1,2\}\) be the three-element semilattice monoid with multiplication \(rs=\max\{r,s\}\), and let
\[
 \mathbf D=\mathcal P(T)
\]
be its complex residuated lattice.
Thus multiplication is setwise multiplication, and
\[
 X\backslash Y=\{z\in T:X\{z\}\subseteq Y\},
 \qquad
 X/Y=\{z\in T:\{z\}Y\subseteq X\}.
\]
The positive idempotents are
\[
 d_0=\{0\},\qquad h_1=\{0,1\},\qquad h_2=\{0,2\},\qquad k=T.
\]

The primitive \(\tau\)-partition is multiplication-coherent, but it is not residuation-coherent.
Indeed,
\[
 \tau(\{1\})=h_1,
 \qquad
 \tau(\{0\})=d_0,
\]
whereas
\[
 \{1\}\backslash\{0\}=\emptyset,
 \qquad
 \tau(\emptyset)=k.
\]
Thus residual coherence forces \(h_1\sim k\).
Similarly, \(\{0,2\}\backslash\{0\}=\emptyset\) forces \(h_2\sim k\).
After the identifications \(h_1\sim k\) and \(h_2\sim k\), let
\[
 E=\{d_0\},
 \qquad
 U=\{h_1,h_2,k\}.
\]
If both inputs have \(\tau\)-value in \(E\), then both inputs are \(\{0\}\), and the product and both residuals again have \(\tau\)-value \(d_0\).
If at least one input has \(\tau\)-value in \(U\), then the product and both residuals have \(\tau\)-value in \(U\).
Indeed, using
\[
 X\backslash Y=\{z:X\{z\}\subseteq Y\},
 \qquad
 X/Y=\{z:\{z\}Y\subseteq X\},
\]
one checks the following two-block output table for each of the three operations \(XY\), \(X\backslash Y\), and \(X/Y\):
\[
\begin{array}{c|cc}
 & E & U\\
\hline
E & E & U\\
U & U & U
\end{array}
\]
where the row records the \(\tau\)-block of \(X\) and the column records the \(\tau\)-block of \(Y\).
Thus no product or residual witness forces \(d_0\) into the block \(\{h_1,h_2,k\}\), and hence
\[
 \Cr(\mathbf D)=\{\{d_0\},\{h_1,h_2,k\}\}.
\]
Thus the residual coherentization performs a genuine collapse.

We now show that this collapse is not explained by steadiness over a nontrivial subsemilattice of the skeleton.
Let \(I\) be a subsemilattice of \(\{d_0,h_1,h_2,k\}\) over which \(\mathbf D\) is subsemilattice-steady in the sense of \cite{BGJPS}.
Since \(\mathbf D\) is a monoid with identity \(d_0\), the \(I\)-local unit of \(d_0\) is \(d_0\), so \(d_0\in I\).
Assume that \(I\) is nontrivial, and let \(t\) be the greatest element of \(I\).
Then \(t\in\{h_1,h_2,k\}\).

If \(t=h_1\), put \(a=\{0\}\) and \(b=\{0,2\}\).
Then \(u_a=d_0\) and \(u_b=d_0\), since \(h_1b\ne b\).
But
\[
 b\backslash a=\emptyset,
\]
and every element of \(I\) fixes \(\emptyset\), so \(u_{b\backslash a}=u_\emptyset=t=h_1\).
Therefore
\[
 u_{b\backslash a}=h_1\ne d_0=u_au_b,
\]
contradicting the residual steadiness law.

If \(t=h_2\) or \(t=k\), put \(a=\{0\}\) and \(b=\{1\}\).
Again
\[
 b\backslash a=\emptyset,
\]
so \(u_{b\backslash a}=u_\emptyset=t\).
On the other hand \(u_a=d_0\), and \(u_b\) is either \(d_0\) or \(h_1\), according as \(h_1\notin I\) or \(h_1\in I\).
In both cases
\[
 u_au_b<t.
\]
Hence
\[
 u_{b\backslash a}\ne u_au_b,
\]
again contradicting the residual steadiness law.

Thus \(\mathbf D\) is not steady over any nontrivial subsemilattice of its positive-idempotent skeleton.
The only subsemilattice-steady presentation is the trivial one \(I=\{d_0\}\), which collapses the whole positive-idempotent skeleton before the decomposition begins.
By contrast, \(\Cr(\mathbf D)\) is canonical and strictly finer: it keeps \(\{d_0\}\) separate and collapses only the residually forced block \(\{h_1,h_2,k\}\).
\end{example}

\subsection{Bridge to the frontier comparison}

The two examples together prepare the general comparison in the next section.
The product example shows how a chosen subsemilattice-steady presentation can hide distinctions at the root by projecting the full positive-idempotent skeleton onto a smaller visible skeleton.
For the specific visible skeleton used there, the obstructing \(\mathbf D\)-coordinate is not available as a further visible skeleton inside that presentation.
The max-semilattice example removes the remaining ambiguity: it shows that the \(\mathbf D\)-factor itself admits no nontrivial subsemilattice-steady presentation, although residual coherentization still gives a proper canonical two-block decomposition.

The residual coherentization exposes, at the first residual level, all skeleton distinctions not forced away by residuation coherence itself.
This is the mechanism behind the general frontier comparison proved in the next section.
There the point is the difference between chosen visibility and canonical coarsening.
The next section proves that residual coherentization refines every subsemilattice-steady decomposition arising from a chosen visibility quotient.
Under iteration, the \(\tau\)-residual frontier refines the corresponding subsemilattice-steady frontier at every finite stage.
Thus the residual method uses no prior choice of visible subsemilattice and loses no distinction except those forced by the product and residual \(\tau\)-value laws.

\section{Residual coherentization refines subsemilattice-steady visibility}
\label{sec:Cr-refines-subsemilattice-steady}

We now make precise the comparison suggested by the preceding examples.
The issue is the difference between chosen visibility and canonical forced coarsening.
A subsemilattice-steady presentation starts by choosing a visible subsemilattice of the positive-idempotent skeleton and reads each local unit through the corresponding projected local unit.
Residual coherentization does not make such a choice.
It starts from the full positive-idempotent skeleton and identifies only those idempotents which are forced to be identified by the product and residual \(\tau\)-value laws.

The section has three steps.
First we prove the comparison at the root: every subsemilattice-steady visibility quotient is already a \(\tau\)-residuation-coherent quotient of the full positive-idempotent skeleton, and hence lies above \(\Cr(\mathbf A)\) in the refinement order.
Then we isolate the restriction mechanism needed for iteration.
Finally we compare the iterated constructions frontierwise.
Terminal components are carried forward in the frontier, so the comparison is not a comparison of graph-theoretic depth levels.

\subsection{First-level visibility quotients}

\begin{definition}[Visibility partition induced by a subsemilattice]
\label{def:visibility-partition}
Let \(\mathbf A\) be a balanced residuated partially ordered semigroup, and let
\[
 I\subseteq \Sk(\mathbf A)
\]
be a nonempty subsemilattice satisfying the visibility condition used in subsemilattice-steady presentations: for every \(r\in\Sk(\mathbf A)\), the maximum
\[
 \pi_I(r)=\max\{p\in I:p\le r\}
\]
exists.
Such a subsemilattice \(I\) will be called \emph{visible}.
The \emph{visibility partition induced by \(I\)} is the partition \(\mathcal P_I\) of \(\Sk(\mathbf A)\) into the fibres of \(\pi_I\):
\[
 r\equiv_I s
 \quad\Longleftrightarrow\quad
 \pi_I(r)=\pi_I(s).
\]
\end{definition}

\begin{definition}[Subsemilattice-steady visibility data]
\label{def:subsemilattice-steady-visibility-data}
Let \(\mathbf A\) be a balanced residuated partially ordered semigroup, and let \(I\subseteq\Sk(\mathbf A)\) be a visible subsemilattice in the sense of Definition~\ref{def:visibility-partition}.
For \(a\in A\), put
\[
 u_a=\pi_I(\tau(a)).
\]
Following \cite{BGJPS} we say that \(\mathbf A\) is \emph{subsemilattice-steady over \(I\)} if, for all \(a,b\in A\),
\[
 u_{ab}=u_au_b,
 \qquad
 u_{a\backslash b}=u_au_b,
 \qquad
 u_{a/b}=u_au_b.
\]
The associated visibility fibres are
\[
 A^I_p=\{a\in A:u_a=p\},
 \qquad p\in I.
\]
For the iterated frontier comparison, a subsemilattice-steady decomposition step means these data together with the standard component condition that the nonempty fibres \(A^I_p\), equipped with the inherited order, multiplication, and residuals, are induced component subalgebras of \(\mathbf A\).
The first-level refinement theorem below uses only the three displayed local-unit identities; the frontier theorem uses the induced-component condition.
This is the subsemilattice-steady visibility condition from~\cite{BGJPS} in the form used below.
The comparison theorems that follow use only the displayed local-unit identities and the induced-component condition stated here; no other result from \cite{BGJPS} is used.
\end{definition}

\begin{theorem}[First-level comparison with subsemilattice-steady decompositions]
\label{thm:Cr-refines-subsemilattice-steady}
Let \(\mathbf A\) be a balanced residuated partially ordered semigroup, and let
\[
 I\subseteq \Sk(\mathbf A)
\]
be a nonempty visible subsemilattice over which \(\mathbf A\) is subsemilattice-steady in the sense of Definition~\ref{def:subsemilattice-steady-visibility-data}.
Then the visibility partition \(\mathcal P_I\) is \(\tau\)-residuation-coherent.
Consequently,
\[
 \Cr(\mathbf A)\le \mathcal P_I,
\]
where \(\le\) denotes refinement of partitions.
Equivalently, every component of the first-level \(\Cr\)-decomposition is contained in a component of the subsemilattice-steady decomposition over \(I\).
In particular, if \(\Sk(\mathbf A)\) is finite, then
\[
 |\Cr(\mathbf A)|\ge |\mathcal P_I|=|I|.
\]
\end{theorem}

\begin{proof}
Let \(a\in A\).
By Definition~\ref{def:subsemilattice-steady-visibility-data},
\[
 u_a=\pi_I(\tau(a)).
\]
Equivalently,
\[
 u_a=\max\{p\in I:p\le \tau(a)\}.
\]
Since
\[
 \tau(a)=a\backslash a=a/a,
\]
the condition \(p\le \tau(a)\) is equivalent to \(pa=a=ap\).
Indeed, \(pa=a\) implies \(p\le a/a\) by residuation, and conversely \(p\le a/a\) gives \(pa\le a\), while positivity of \(p\) gives \(a\le pa\).
The right-sided condition is analogous.

We show that \(\mathcal P_I\) is \(\tau\)-residuation-coherent.
Let \(a,b\in A\).
Since \(\mathbf A\) is subsemilattice-steady over \(I\), the \(I\)-local units satisfy
\[
 u_{ab}=u_au_b,
 \qquad
 u_{a\backslash b}=u_au_b,
 \qquad
 u_{a/b}=u_au_b.
\]
Using \(u_x=\pi_I(\tau(x))\), these identities become
\[
 \pi_I(\tau(ab))=
 \pi_I(\tau(a))\,\pi_I(\tau(b)),
\]
\[
 \pi_I(\tau(a/b))=
 \pi_I(\tau(a))\,\pi_I(\tau(b)),
\]
and
\[
 \pi_I(\tau(a\backslash b))=
 \pi_I(\tau(a))\,\pi_I(\tau(b)).
\]
Now let \(P,Q\) be two blocks of \(\mathcal P_I\), and let \(a,a'\in\layer{P}\) and \(b,b'\in\layer{Q}\).
Then
\[
 \pi_I(\tau(a))=\pi_I(\tau(a')),
 \qquad
 \pi_I(\tau(b))=\pi_I(\tau(b')).
\]
The three displayed identities therefore show that the \(\mathcal P_I\)-block of each of
\[
 \tau(ab),\qquad \tau(a\backslash b),\qquad \tau(a/b)
\]
is the same as the corresponding block obtained from \(a'\) and \(b'\).
Equivalently, for each pair of input blocks \(P,Q\), all product and residual \(\tau\)-outputs lie in the single \(\mathcal P_I\)-block represented by
\[
 \pi_I(\tau(a))\,\pi_I(\tau(b)).
\]
Thus \(\mathcal P_I\) is \(\tau\)-residuation-coherent.

By Proposition~\ref{prop:canonicality-Cr}, \(\Cr(\mathbf A)\) is the finest \(\tau\)-residuation-coherent partition of \(\Sk(\mathbf A)\).
Since \(\mathcal P_I\) is such a partition, we obtain
\[
 \Cr(\mathbf A)\le \mathcal P_I.
\]

It remains only to translate this skeleton refinement into the component statement.
If \(C\in\Cr(\mathbf A)\), then \(C\) is contained in a unique \(\mathcal P_I\)-block, say \(\pi_I^{-1}(p)\).
Therefore every element \(x\) in the \(\Cr\)-component
\[
 A_C^{\Cr}:=\{x\in A:\tau(x)\in C\}
\]
satisfies
\[
 u_x=\pi_I(\tau(x))=p.
\]
Hence \(A_C^{\Cr}\) is contained in the subsemilattice-steady fibre
\[
 A_p^I=\{x\in A:u_x=p\}.
\]
This proves the component containment.
Finally, each \(p\in I\) lies in its own fibre, because \(\pi_I(p)=p\).
Thus \(\mathcal P_I\) has exactly \(|I|\) blocks.
If \(\Sk(\mathbf A)\) is finite, refinement gives
\[
 |\Cr(\mathbf A)|\ge |\mathcal P_I|=|I|.
\]
\end{proof}

\begin{remark}[Projection visibility and quotient coherentization]
\label{rem:meaning-Cr-refines-subsemilattice}
The theorem does not say that the two decompositions choose the same components or that there is a canonical bijection between their nodes.
It says that the \(\Cr\)-partition is a refinement of the partition induced by the chosen visible subsemilattice.
Thus the \(\Cr\)-decomposition cannot have fewer first-level components than any subsemilattice-steady presentation.

The reason is structural.
A subsemilattice-steady presentation first chooses a visible subsemilattice
\[
 I\subseteq \Sk(\mathbf A)
\]
and replaces the full local unit \(\tau(a)\) by its visible projection
\[
 u_a=\pi_I(\tau(a)).
\]
It then asks the product and residual local-unit laws to hold after this projection.
Thus the construction keeps the chosen visible part \(I\) and reads the remaining local-unit data through the projection onto \(I\).
In this sense it is projection-like: it is close to passing to the fixed part of an interior operator, or conucleus.

Residual coherentization proceeds in the quotient direction.
It starts with the full positive-idempotent skeleton and then canonically identifies only those idempotents which are forced to be identified by the product and residual \(\tau\)-value laws.
Thus \(\Cr(\mathbf A)\) is the finest quotient of the full skeleton on which the three trigger laws
\[
 xy,\qquad x\backslash y,\qquad x/y
\]
have well-defined \(\tau\)-values at the block level.
This is why every subsemilattice-steady visibility quotient lies above \(\Cr(\mathbf A)\) in the refinement order.
\end{remark}

\subsection{Restriction to induced components}

\begin{definition}[Induced component subalgebra]
\label{def:induced-component-subalgebra}
Let \(\mathbf B\) be a balanced residuated partially ordered semigroup.
An \emph{induced component subalgebra} of \(\mathbf B\) is a balanced residuated partially ordered semigroup \(\mathbf C\) whose carrier \(C\) is a subset of \(B\), and whose order, multiplication, and residuals are the restrictions of the corresponding order and operations of \(\mathbf B\).
We also require
\[
 \Sk(\mathbf C)=\Sk(\mathbf B)\cap C .
\]
Then the local-unit map of \(\mathbf C\) is automatically the restriction of the local-unit map of \(\mathbf B\), since for \(x\in C\),
\[
 \tau_{\mathbf C}(x)
 =
 x\backslash_{\mathbf C}x
 =
 x\backslash_{\mathbf B}x
 =
 \tau_{\mathbf B}(x).
\]
\end{definition}

\begin{lemma}[Residual-coherent components are induced]
\label{lem:coherent-components-induced}
Let \(\Omega\) be a \(\tau\)-residuation-coherent skeleton partition of \(\Sk(\M)\), and let \(A\in\Omega\).
Then the component algebra \(\Comp{A}\) is an induced component subalgebra of \(\M\) in the sense of Definition~\ref{def:induced-component-subalgebra}.
More generally, the same statement holds inside any induced component subalgebra.
\end{lemma}

\begin{proof}
By definition, the order, multiplication, and residuals of \(\Comp{A}\) are the restrictions of the corresponding order and operations of \(\M\) to \(\layer{A}\).
The required skeleton identity
\[
 \Sk(\Comp{A})=\Sk(\M)\cap\layer{A}
\]
is Proposition~\ref{prop:component-algebras}.
The same argument applies after replacing \(\M\) by any induced component subalgebra.
\end{proof}

\begin{lemma}[Nested induced component subalgebras]
\label{lem:nested-induced-component-subalgebras}
Let \(\mathbf B\) be a balanced residuated partially ordered semigroup.
If \(\mathbf R\) and \(\mathbf C\) are induced component subalgebras of \(\mathbf B\) and \(R\subseteq C\), then \(\mathbf R\) is an induced component subalgebra of \(\mathbf C\).
\end{lemma}

\begin{proof}
The order, multiplication, and residuals on \(\mathbf R\) and on \(\mathbf C\) are all inherited from \(\mathbf B\), so the operations of \(\mathbf R\) are also the restrictions of the operations of \(\mathbf C\).
For the skeletons,
\[
 \Sk(\mathbf R)
 =
 \Sk(\mathbf B)\cap R
 =
 \bigl(\Sk(\mathbf B)\cap C\bigr)\cap R
 =
 \Sk(\mathbf C)\cap R .
\]
Thus \(\mathbf R\) is induced in \(\mathbf C\).
\end{proof}

\begin{lemma}[Restriction of \(\tau\)-residuation coherence]
\label{lem:restriction-of-tau-res-coherence}
Let \(\mathbf B\) be a balanced residuated partially ordered semigroup, and let \(\mathbf C\) be an induced component subalgebra of \(\mathbf B\).
If \(\mathcal P\) is a \(\tau\)-residuation-coherent partition of \(\Sk(\mathbf B)\), then
\[
 \mathcal P\!\upharpoonright C
 =
 \{\,P\cap\Sk(\mathbf C):P\in\mathcal P,
        P\cap\Sk(\mathbf C)\ne\varnothing\,\}
\]
is a \(\tau\)-residuation-coherent partition of \(\Sk(\mathbf C)\).
\end{lemma}

\begin{proof}
Since \(\mathbf C\) is induced, multiplication and both residuals in \(\mathbf C\) are the restrictions of the corresponding operations of \(\mathbf B\).
The local-unit map is restricted as well, by Definition~\ref{def:induced-component-subalgebra}.
Let \(P,Q\) be blocks of \(\mathcal P\!\upharpoonright C\), and let \(P^+,Q^+\in\mathcal P\) be the corresponding ambient blocks, so that
\[
 P=P^+\cap\Sk(\mathbf C),
 \qquad
 Q=Q^+\cap\Sk(\mathbf C).
\]
Take \(x,y\in C\) with
\[
 \tau_{\mathbf C}(x)\in P,
 \qquad
 \tau_{\mathbf C}(y)\in Q.
\]
The three trigger elements
\[
 xy,\qquad x\backslash y,\qquad x/y
\]
are computed in \(\mathbf C\) exactly as in \(\mathbf B\), and their \(\tau\)-values are the corresponding ambient \(\tau\)-values.
By \(\tau\)-residuation coherence of \(\mathcal P\), these ambient \(\tau\)-values lie in the unique ambient output block determined by \(P^+\) and \(Q^+\).
Since the three trigger elements belong to \(C\), the same \(\tau\)-values lie in \(\Sk(\mathbf C)\), and therefore in the corresponding restricted output block of \(\mathcal P\!\upharpoonright C\).
Thus \(\mathcal P\!\upharpoonright C\) is \(\tau\)-residuation-coherent.
\end{proof}

\begin{corollary}[Hereditary visibility]
\label{cor:hereditary-visibility}
Let \(\mathbf C\) be an induced component subalgebra of a balanced residuated partially ordered semigroup, and let \(\mathcal P_C\) be the visibility partition used for a subsemilattice-steady step inside \(\mathbf C\) in the sense of Definition~\ref{def:subsemilattice-steady-visibility-data}.
If \(\mathbf R\) is an induced component subalgebra of \(\mathbf C\), then
\[
 \mathcal P_C\!\upharpoonright R
\]
is a \(\tau\)-residuation-coherent partition of \(\Sk(\mathbf R)\).
In particular, this applies when \(\mathbf R\) is a residual-coherent component of \(\mathbf C\).
\end{corollary}

\begin{proof}
By Theorem~\ref{thm:Cr-refines-subsemilattice-steady}, applied inside \(\mathbf C\), the visibility partition \(\mathcal P_C\) is \(\tau\)-residuation-coherent on \(\Sk(\mathbf C)\).
If \(\mathbf R\) is induced in \(\mathbf C\), Lemma~\ref{lem:restriction-of-tau-res-coherence} applies directly.
If \(\mathbf R\) is a residual-coherent component of \(\mathbf C\), then Lemma~\ref{lem:coherent-components-induced}, applied inside \(\mathbf C\), shows that \(\mathbf R\) is an induced component subalgebra of \(\mathbf C\), and the same restriction lemma applies.
\end{proof}

\subsection{Frontier comparison under iteration}

\begin{definition}[Frontier partitions of an iterated decomposition]
\label{def:frontier-partitions}
Let a decomposition procedure be iterated on a balanced residuated partially ordered semigroup \(\mathbf A\).
A component is \emph{terminal} for the procedure if one more application of the procedure produces no proper refinement of that component.
The \emph{frontier partitions}
\[
 \mathcal F_0,\mathcal F_1,\mathcal F_2,\ldots
\]
are defined as follows.
Put
\[
 \mathcal F_0=\{A\}.
\]
Given \(\mathcal F_n\), each nonterminal component is replaced by the components produced by one application of the decomposition procedure inside it, while each terminal component is carried forward unchanged.
The resulting partition of the original carrier is \(\mathcal F_{n+1}\).

Thus \(\mathcal F_n\) is not the set of vertices at graph-theoretic depth \(n\).
It is the current frontier after \(n\) rounds of refinement.
In particular, terminal leaves remain present in all later frontier partitions.
\end{definition}

For the residual coherentization procedure, Proposition~\ref{prop:tau-residuation-cohesive-terminal} gives this terminality a concrete intrinsic form: a residual frontier component is terminal exactly when its induced component algebra is \(\tau\)-residuation-cohesive.

\begin{theorem}[Frontier comparison for the residual tree]
\label{thm:fibrewise-tree-comparison}
Let \(\mathbf A\) be a balanced residuated partially ordered semigroup.
Let
\[
 (\mathcal R_n)_{n\ge 0}
 \qquad\text{and}\qquad
 (\mathcal S_n)_{n\ge 0}
\]
be frontier partitions of the original carrier.
Assume that \((\mathcal R_n)_{n\ge 0}\) is obtained by iterating the residual coherentization \(\Cr\).
Assume that \((\mathcal S_n)_{n\ge 0}\) is obtained by iterating the following defined visibility step.
For each nonterminal induced component \(\mathbf B\), choose a nonempty visible subsemilattice
\[
 I_{\mathbf B}\subseteq\Sk(\mathbf B)
\]
such that the induced visibility data
\[
 \pi_{I_{\mathbf B}},\qquad
 u_b=\pi_{I_{\mathbf B}}(\tau_{\mathbf B}(b)),\qquad
 B^{I_{\mathbf B}}_p=\{b\in B:u_b=p\}
\]
satisfy Definition~\ref{def:subsemilattice-steady-visibility-data} inside \(\mathbf B\), and such that the nonempty fibres \(B^{I_{\mathbf B}}_p\) are induced component subalgebras.
Then replace \(\mathbf B\) by these induced component subalgebras, carrying terminal components forward unchanged.
Then
\[
 \mathcal R_n\le \mathcal S_n
 \qquad(n\ge 0).
\]
Consequently, whenever the frontier partitions are finite,
\[
 |\mathcal R_n|\ge |\mathcal S_n|
 \qquad(n\ge 0).
\]
If the subsemilattice-steady iteration terminates with terminal frontier \(\mathcal S_N\), and the residual iteration terminates with terminal frontier \(\mathcal R_M\), then
\[
 \mathcal R_M\le \mathcal S_N .
\]
By Proposition~\ref{prop:tau-residuation-cohesive-terminal}, the induced component algebras on the members of \(\mathcal R_M\) are precisely the terminal \(\tau\)-residuation-cohesive residual components.
Thus every terminal residual component is contained in a unique terminal subsemilattice-steady component, and every terminal subsemilattice-steady component is the disjoint union of the terminal \(\tau\)-residuation-cohesive residual components contained in it.
In particular,
\[
 |\mathcal R_M|\ge |\mathcal S_N|.
\]
Each terminal \(\tau\)-residuation-cohesive residual component is an induced balanced residuated component subalgebra of the terminal subsemilattice-steady component which contains it.
\end{theorem}

\begin{proof}
We prove the refinement statement by induction on \(n\).
For \(n=0\),
\[
 \mathcal R_0=\mathcal S_0=\{A\},
\]
so the assertion is trivial.

Assume
\[
 \mathcal R_n\le\mathcal S_n .
\]
Let \(R'\in\mathcal R_{n+1}\).
Let \(R\in\mathcal R_n\) be the parent of \(R'\); thus either \(R\) is terminal and \(R'=R\), or \(R\) is nonterminal and \(R'\) is one of the components obtained by applying \(\Cr\) inside \(R\).
By the induction hypothesis, choose \(S\in\mathcal S_n\) such that
\[
 R\subseteq S.
\]

If \(S\) is terminal for the subsemilattice-steady procedure, then \(S\) is carried forward unchanged to \(\mathcal S_{n+1}\), and
\[
 R'\subseteq R\subseteq S.
\]
Hence \(R'\) is contained in a member of \(\mathcal S_{n+1}\).

Suppose now that \(S\) is split at the next subsemilattice-steady step, and let \(\mathcal P_S\) be the visibility partition used inside \(S\).
We first note that \(R\) is an induced component subalgebra of \(S\).
Residual-coherent components are induced in their parent components by Lemma~\ref{lem:coherent-components-induced}, and repeated application of Lemma~\ref{lem:nested-induced-component-subalgebras} shows that every residual frontier component is induced in the original algebra.
Similarly, by the definition of the subsemilattice-steady decomposition used here, each subsemilattice-steady frontier component is an induced component subalgebra of the original algebra; nesting preserves inducedness by Lemma~\ref{lem:nested-induced-component-subalgebras}.
Since \(R\subseteq S\), Lemma~\ref{lem:nested-induced-component-subalgebras} therefore gives that \(R\) is induced in \(S\).
Let \(\mathbf R\) and \(\mathbf S\) denote the induced component subalgebras with carriers \(R\) and \(S\), respectively.
By Corollary~\ref{cor:hereditary-visibility},
\[
 \mathcal P_S\!\upharpoonright R
\]
is a \(\tau\)-residuation-coherent partition of \(\Sk(\mathbf R)\).

If \(R\) is terminal for the residual procedure, equivalently if \(\mathbf R\) is \(\tau\)-residuation-cohesive, then \(\Cr(\mathbf R)\) is the one-block partition of \(\Sk(\mathbf R)\).
Since \(\Cr(\mathbf R)\) is the finest \(\tau\)-residuation-coherent partition of \(\Sk(\mathbf R)\), the restricted visibility partition \(\mathcal P_S\!\upharpoonright R\) must also be one-block.
Therefore \(R'=R\) lies inside a single child of \(S\).

If \(R\) is nonterminal for the residual procedure, then \(R'\) is a \(\Cr(\mathbf R)\)-component.
Since \(\mathcal P_S\!\upharpoonright R\) is \(\tau\)-residuation-coherent, Proposition~\ref{prop:canonicality-Cr} gives
\[
 \Cr(\mathbf R)\le \mathcal P_S\!\upharpoonright R .
\]
Hence \(R'\) is contained in one block of \(\mathcal P_S\!\upharpoonright R\), and therefore in one child of \(S\) in \(\mathcal S_{n+1}\).

In all cases, every member of \(\mathcal R_{n+1}\) is contained in a member of \(\mathcal S_{n+1}\).
Thus
\[
 \mathcal R_{n+1}\le\mathcal S_{n+1}.
\]
The induction is complete.

The cardinality inequality follows from refinement of finite partitions.

Assume now that the subsemilattice-steady iteration terminates at \(\mathcal S_N\).
Then
\[
 \mathcal S_m=\mathcal S_N
 \qquad(m\ge N),
\]
because terminal components are carried forward unchanged.
If the residual iteration terminates at \(\mathcal R_M\), then every induced algebra on a member of \(\mathcal R_M\) is \(\tau\)-residuation-cohesive by Proposition~\ref{prop:tau-residuation-cohesive-terminal}.
Replacing \(M\) by a larger index if necessary, we may assume \(M\ge N\).
The refinement already proved gives
\[
 \mathcal R_M\le \mathcal S_M=\mathcal S_N .
\]
Thus every terminal residual component, equivalently every terminal \(\tau\)-residuation-cohesive residual component, is contained in a unique terminal subsemilattice-steady component.
Since both frontiers are partitions of the original carrier, each terminal subsemilattice-steady component is the disjoint union of the terminal \(\tau\)-residuation-cohesive residual components contained in it.
The cardinality inequality
\[
 |\mathcal R_M|\ge |\mathcal S_N|
\]
follows immediately.
Finally, both the terminal \(\tau\)-residuation-cohesive residual component and the terminal subsemilattice-steady component are induced in the original algebra, by the same argument as above.
Since the former is contained in the latter, Lemma~\ref{lem:nested-induced-component-subalgebras} shows that each terminal \(\tau\)-residuation-cohesive residual component is an induced balanced residuated component subalgebra of the terminal subsemilattice-steady component containing it.
\end{proof}

The conclusion is therefore stronger than a comparison of frontier widths.
At every finite stage, and in particular at terminality when both procedures terminate, the \(\tau\)-residual frontier refines the corresponding subsemilattice-steady frontier.
Thus each terminal subsemilattice-steady component is assembled from terminal \(\tau\)-residuation-cohesive residual components.

\subsection{Frontier refinement and rooted-visibility obstruction}
\label{subsec:frontier-refinement-and-rooted-visibility-obstruction}

The preceding theorem gives the formal frontier comparison.
We now record the structural reason behind it and separate it from the stronger rooted-visibility obstruction exhibited by the examples.
There are two distinct points.

\smallskip
First, the residual-coherent construction refines the subsemilattice-steady method at the level of iterated frontiers.
This is the content of Theorem~\ref{thm:fibrewise-tree-comparison}: the residual-coherent frontier refines the corresponding subsemilattice-steady frontier at every finite stage of the decomposition tree.
The conceptual reason is that \(\tau\)-residual coherence identifies only those positive-idempotent distinctions which must be identified in order for the product and residual \(\tau\)-value laws to be well defined at quotient level.
The formal reason is that a subsemilattice-steady presentation starts from a visible subsemilattice \(I\subseteq\Sk(\mathbf A)\) and replaces the full local unit \(r\) by its visible projection
\[
 \pi_I(r)=\max\{p\in I:p\le r\}.
\]
The steadiness assumptions say precisely that, after this projection, the product and residual \(\tau\)-value laws are well defined.
Thus the visibility partition induced by \(I\) is already a \(\tau\)-residuation-coherent quotient of the full positive-idempotent skeleton.
Since \(\Cr(\mathbf A)\) is the finest such quotient at the component to which it is applied, the residual-coherent step refines the corresponding subsemilattice-steady step.
Theorem~\ref{thm:fibrewise-tree-comparison} iterates this one-step refinement along the whole decomposition tree.

\smallskip
Second, the examples show more than frontier refinement.
Examples~\ref{ex:C2_times_max3_residual_visibility} and~\ref{ex:max3_no_nontrivial_subsemilattice_steady} show a genuine separation between forced residual coherence and \lq rooted\rq\ subsemilattice visibility.
The first example shows that residual coherentization can give a strictly finer decomposition than subsemilattice visibility.
The second shows the stronger phenomenon: an algebra may be indecomposable by all nontrivial subsemilattice-steady presentations, while still admitting a proper \(\tau\)-residual coherent decomposition.

The admissible quotients in the subsemilattice-steady method are not arbitrary coherent quotients.
They must be representable as fibres of a visibility projection.
Such fibres have a rooted-convex shape.
If \(F_p=\pi_I^{-1}(p)\), then \(p\in F_p\), and \(p\le r\) for every \(r\in F_p\).
Moreover, \(F_p\) is order-convex: if \(r,s\in F_p\) and \(r\le t\le s\), then any visible element above \(p\) and below \(t\) would also lie below \(s\), which would contradict \(\pi_I(s)=p\).
Hence every visibility fibre is rooted convex.

Example~\ref{ex:max3_no_nontrivial_subsemilattice_steady} illustrates the difference.
There
\[
 \Cr(\mathbf D)=\{\{d_0\},\{h_1,h_2,k\}\}.
\]
The block \(\{h_1,h_2,k\}\) is not problematic because of non-convexity; in the diamond order it is convex.
The obstruction is rootedness.
It has two incomparable minimal elements, \(h_1\) and \(h_2\), and no least element.
Therefore it cannot be a visibility fibre at any stage of an iterated subsemilattice-steady decomposition.
Indeed, in that example every nontrivial visible subsemilattice still violates a residual \(\tau\)-value law, while the only subsemilattice-steady visibility is the trivial one.
Thus the example exhibits collapse by forced join-congruence beyond rooted subsemilattice-visibility representability.

This rooted-convex restriction is {\em inherited} by restriction to components.
Indeed, suppose that \(F\) is rooted convex in a poset \(P\), and that \(G\subseteq F\) is rooted convex in the induced order on \(F\).
Then \(G\) is rooted convex in \(P\): its root is still below all elements of \(G\), and if \(a,b\in G\) and \(a\le t\le b\) in \(P\), then \(t\in F\) by the convexity of \(F\), and then \(t\in G\) by the relative convexity of \(G\) inside \(F\).

Consequently, every proper block produced by a subsemilattice-steady visibility step is rooted convex in the skeleton of the component in which it is produced, and this rooted-convex property is preserved when one passes to later induced subcomponents.
In particular, if a proposed block is not rooted convex in the original positive-idempotent skeleton, and all previous visibility blocks containing it are rooted convex there, then that proposed block cannot occur at any later finite stage of an iterated visibility tree.
Thus a forced coherent block with no rooted-convex representation cannot be recovered merely by first passing to a coarser rooted visibility component and then decomposing inside that component.

Residual coherentization is not subject to this rooted-visibility representability constraint.
It forms the forced join-congruence quotient dictated by the product and residual \(\tau\)-value laws.
The resulting coherent blocks need not be fibres of any visibility projection onto a chosen subsemilattice, neither at the root nor after restriction to components.

The product example, Example~\ref{ex:C2_times_max3_residual_visibility}, explains why the comparison must be made along the whole decomposition tree, not only at the root.
There the subsemilattice-steady method obtains a nontrivial first split, but it does so by hiding the obstructing \(D\)-coordinate inside the larger visible fibres
\[
 A_{(p,d_0)}=B_p\times\mathcal P(T),
 \qquad p\in\{u,g\}.
\]
This creates the apparent possibility that the missing residual split might be recovered later inside those coarser components.
The residual coherentization, however, would require the two upper blocks
\[
 U_u=\{(u,h_1),(u,h_2),(u,k)\},
 \qquad
 U_g=\{(g,h_1),(g,h_2),(g,k)\}.
\]
Both are convex in the original positive-idempotent skeleton \(\{u,g\}\times\{d_0,h_1,h_2,k\}\), but neither is rooted.
Indeed, \(U_u\) has two incomparable minimal elements, \((u,h_1)\) and \((u,h_2)\), and \(U_g\) has two incomparable minimal elements, \((g,h_1)\) and \((g,h_2)\).
By the rooted-convex inheritance argument, such blocks cannot occur at any finite stage of an iterated subsemilattice-steady visibility tree.
Thus the missing residual split cannot be recovered later merely by decomposing inside the coarser first-level visibility fibres.

\medskip
It is important that the obstruction here is rootedness, not merely the fact that a first-level visibility quotient is coarse.
In Example~6.11 of~\cite{BGJPS}, the first visibility over \(I=\{1,p\}\) identifies the full local units \(p\) and \(q\), but the block \(\{p,q\}\) is rooted convex, with root \(p\).
Inside the upper fibre \(\{p,q,b,\bot\}\), the later split into \(\{p\}\) and \(\{q,b,\bot\}\) is again a rooted visibility split.
Thus that example is compatible with the frontier theorem: the hidden distinction is recoverable later precisely because no non-rooted visibility obstruction is present.

\bigskip
The subsemilattice-steady method deliberately begins with a chosen visibility quotient.
The point proved here is different.
In the frontierwise refinement sense of Theorem~\ref{thm:fibrewise-tree-comparison}, residual coherentization refines every subsemilattice-steady decomposition arising from such a visibility choice.
The rooted-visibility examples show a stronger separation: residual coherentization can produce forced coherent blocks which are not representable as visibility fibres, and it can even produce a proper decomposition when no nontrivial subsemilattice-steady presentation is available.

\section{Conclusion}
\label{sec:conclusion}

This paper gives a canonical decomposition--reconstruction theorem for balanced residuated partially ordered semigroups.
The starting point is the intrinsic local-unit map \(\tau\).
The primitive fibres of \(\tau\) are often too fine: product and residual outputs need not have local units determined by the local units of their inputs.
The residual coherentization \(\Cr(\M)\) is the canonical response to this failure.
It is the finest quotient of the positive-idempotent skeleton on which the local-unit behavior of multiplication and of both residuals becomes coherent.

The reconstruction theorem shows that this quotient loses no information about the original algebra.
Its blocks carry component algebras, and the quotient skeleton records the target component for products and residuals.
The additional transport data are not two ordinary direct systems of transition maps.
Instead, for each comparable source and target component, the construction uses families of shadow maps indexed by the positive idempotents of the target component.
The product-shadow family is already needed for multiplication: for \(x\in\layer{A}\), \(y\in\layer{B}\), and \(C=A\vee B\), the product is recovered as
\[
 xy=\min\{(xr)(yr):r\in C\}.
\]
Thus product recovery is not obtained from a single transition map, but from the whole family of product-shadow maps into the target component.
Order recovery then uses both the product-shadow and residual-shadow families.
Their order-theoretic behavior is different: the product shadow \(xq\) lies above \(x\) in the ambient order, while the residual shadow \(y/q\) lies below \(y\) in the ambient order.
Thus, for a suitable target idempotent \(q\), the comparison
\[
 xq\le y/q
\]
inside the target component is a genuine ambient-order certificate.
After multiplication and order have been recovered, the residuals are forced by residuation.
Thus residual coherentization is not merely a way of producing components; it is a complete canonical encoding of the original balanced residuated ordered semigroup.

The comparison with subsemilattice-steady visibility decompositions shows how the present construction differs from a chosen-visibility approach.
Instead of selecting a visible subsemilattice in advance, residual coherentization uses the canonical quotient of the full positive-idempotent skeleton forced by the product and residual \(\tau\)-value laws.
The frontier comparison theorem makes this precise.
At every finite stage, and in particular at terminality when both procedures terminate, the \(\tau\)-residual frontier refines the corresponding subsemilattice-steady frontier.
Thus each terminal subsemilattice-steady component is assembled from terminal \(\tau\)-residuation-cohesive residual components.

The construction is therefore not tied to a special class of chains or to a particular logical presentation.
It is formulated in the broader language of balanced residuated partially ordered semigroups, a setting which includes the ordered multiplicative and residual structure underlying many familiar residuated lattices, quantales, and ideal-theoretic examples.
In this sense residual coherentization is a general algebraic decomposition mechanism for local-unit-based residuated structures.

\section*{Funding}

The author gratefully acknowledges support from the Ministry of Culture and Innovation of Hungary, through the National Research, Development and Innovation Fund, grant no.~K138596.


\begin{thebibliography}{99}

\bibitem{AglianoMontagna2003}
P.~Aglian\`o and F.~Montagna.
\newblock Varieties of BL-algebras I: general properties.
\newblock \emph{J. Pure Appl. Algebra} \textbf{181} (2003), 105--129.
\newline\url{https://doi.org/10.1016/S0022-4049(02)00329-8}

\bibitem{BGJPSm}
S.~Bonzio, J.~Gil-F\'erez, P.~Jipsen, A.~P\v{r}enosil, and M.~Sugimoto.
\newblock On the structure of balanced residuated partially ordered monoids.
\newblock In U.~Fahrenberg, W.~Fussner, and R.~Gl\"uck, editors,
\emph{Relational and Algebraic Methods in Computer Science},
volume~14787 of \emph{Lecture Notes in Computer Science}, pages 83--100.
Springer, Cham, 2024.
\newline\url{https://doi.org/10.1007/978-3-031-68279-7_6}

\bibitem{BGJPS}
S.~Bonzio, J.~Gil-F\'erez, P.~Jipsen, A.~P\v{r}enosil, and M.~Sugimoto.
\newblock Balanced residuated partially ordered semigroups.
\newblock arXiv:2505.12024v2, 2026.
\newline\url{https://doi.org/10.48550/arXiv.2505.12024}

\bibitem{Busaniche2005}
M.~Busaniche.
\newblock Decomposition of BL-chains.
\newblock \emph{Algebra Univers.} \textbf{52} (2005), 519--525.
\newline\url{https://doi.org/10.1007/s00012-004-1899-4}

\bibitem{CastiglioniZuluaga2021}
J.L.~Castiglioni and W.J.~Zuluaga Botero.
\newblock Split exact sequences of finite MTL-chains.
\newblock \emph{Rev. Un. Mat. Argentina} \textbf{62} (2021), 295--304.
\newline\url{https://doi.org/10.33044/revuma.1787}

\bibitem{CignoliDOttavianoMundici2000}
R.~L. O. Cignoli, I.~M. L. D'Ottaviano, and D.~Mundici.
\newblock \emph{Algebraic Foundations of Many-Valued Reasoning}.
\newblock Trends in Logic, vol.~7, Kluwer Academic Publishers, Dordrecht, 2000.
\newline\url{https://doi.org/10.1007/978-94-015-9480-6}

\bibitem{Clifford1941}
A.H.~Clifford.
\newblock Semigroups admitting relative inverses.
\newblock \emph{Ann. Math.} \textbf{42} (1941), 1037--1049.
\newline\url{https://doi.org/10.2307/1968781}

\bibitem{CliffordPreston1961}
A.H.~Clifford and G.B.~Preston.
\newblock \emph{The Algebraic Theory of Semigroups, Vol. I}.
\newblock American Mathematical Society, Providence, 1961.
\newline\url{https://doi.org/10.1090/surv/007.1}

\bibitem{EstevaGodo2001}
F.~Esteva and L.~Godo.
\newblock Monoidal t-norm based logic: towards a logic for left-continuous t-norms.
\newblock \emph{Fuzzy Sets Syst.} \textbf{124} (2001), 271--288.
\newline\url{https://doi.org/10.1016/S0165-0114(01)00098-7}

\bibitem{GalatosJipsenKowalskiOno2007}
N.~Galatos, P.~Jipsen, T.~Kowalski, and H.~Ono.
\newblock \emph{Residuated Lattices: An Algebraic Glimpse at Substructural Logics}.
\newblock Elsevier, Amsterdam, 2007.
\newline\url{https://doi.org/10.1016/S0049-237X(13)70001-6}

\bibitem{GilFerezJipsenLodhia2023}
J.~Gil-F\'erez, P.~Jipsen, and S.~Lodhia.
\newblock The structure of locally integral involutive po-monoids and semirings.
\newblock In R.~Gl\"uck, L.~Santocanale, and M.~Winter, editors, \emph{Relational and Algebraic Methods in Computer Science}, RAMiCS 2023, Lecture Notes in Computer Science, vol.~13896, Springer, Cham, pp.~69--86, 2023.
\newline\url{https://doi.org/10.1007/978-3-031-28083-2_5}

\bibitem{GilFerezJipsenSugimoto2023}
J.~Gil-F\'erez, P.~Jipsen, and M.~Sugimoto.
\newblock Locally integral involutive PO-semigroups.
\newblock \emph{Fundam. Inform.} \textbf{195} (2025), Article~7.
\newblock First circulated as arXiv:2310.12926 (2023).
\newline\url{https://doi.org/10.46298/fi.12449}
\newline\url{https://doi.org/10.48550/arXiv.2310.12926}

\bibitem{Gilmer1972}
R.~Gilmer.
\newblock \emph{Multiplicative Ideal Theory}.
\newblock Marcel Dekker, New York, 1972.

\bibitem{Green1951}
J.A.~Green.
\newblock On the structure of semigroups.
\newblock \emph{Ann. Math.} \textbf{54} (1951), 163--172.
\newline\url{https://doi.org/10.2307/1969317}

\bibitem{Hajek1998}
P.~H\'ajek.
\newblock \emph{Metamathematics of Fuzzy Logic}.
\newblock Trends in Logic, vol.~4, Kluwer Academic Publishers, Dordrecht, 1998.
\newline\url{https://doi.org/10.1007/978-94-011-5300-3}

\bibitem{Horcik2007}
R.~Hor\v{c}\'ik.
\newblock Structure of commutative cancellative integral residuated lattices on \((0,1]\).
\newblock \emph{Algebra Univers.} \textbf{57} (2007), 303--332.
\newline\url{https://doi.org/10.1007/s00012-007-2050-0}

\bibitem{Horcik2011}
R.~Hor\v{c}\'ik.
\newblock On the structure of finite integral commutative residuated chains.
\newblock \emph{J. Log. Comput.} \textbf{21} (2011), 717--728.
\newline\url{https://doi.org/10.1093/logcom/exp059}

\bibitem{HorcikMontagna2009}
R.~Hor\v{c}\'ik and F.~Montagna.
\newblock Archimedean classes in integral commutative residuated chains.
\newblock \emph{Math. Log. Q.} \textbf{55} (2009), 320--336.
\newline\url{https://doi.org/10.1002/malq.200710091}

\bibitem{Howie1995}
J.M.~Howie.
\newblock \emph{Fundamentals of Semigroup Theory}.
\newblock Oxford University Press, Oxford, 1995.
\newline\url{https://doi.org/10.1093/oso/9780198511946.001.0001}

\bibitem{JeneiGroupRepresentation}
S.~Jenei.
\newblock Group representation for even and odd involutive commutative residuated chains.
\newblock \emph{Stud. Log.} \textbf{110} (2022), 881--922.
\newblock First circulated as arXiv:1910.01404 (2019).
\newline\url{https://doi.org/10.1007/s11225-021-09981-y}
\newline\url{https://doi.org/10.48550/arXiv.1910.01404}

\bibitem{Jipsen2024BalancedPosets}
P.~Jipsen.
\newblock On the structure of balanced residuated posets.
\newblock In \emph{Abstracts of the Topology, Algebra, and Categories in Logic 2024 Conference}, Barcelona, 2024.

\bibitem{JipsenSugimoto2022}
P.~Jipsen and M.~Sugimoto.
\newblock On varieties of residuated po-magmas and the structure of finite ipo-semilattices.
\newblock In \emph{Topology, Algebra, and Categories in Logic 2022, Coimbra -- Book of Abstracts}, 2022.

\bibitem{JipsenTuytValota2021}
P.~Jipsen, O.~Tuyt, and D.~Valota.
\newblock The structure of finite commutative idempotent involutive residuated lattices.
\newblock \emph{Algebra Univers.} \textbf{82}, 57 (2021).
\newline\url{https://doi.org/10.1007/s00012-021-00751-4}

\bibitem{LarsenMcCarthy1971}
M.~D. Larsen and P.~J. McCarthy.
\newblock \emph{Multiplicative Theory of Ideals}.
\newblock Academic Press, New York, 1971.

\bibitem{MontagnaNogueraHorcik2006}
F.~Montagna, C.~Noguera, and R.~Hor\v{c}\'ik.
\newblock On weakly cancellative fuzzy logics.
\newblock \emph{J. Log. Comput.} \textbf{16} (2006), 423--450.
\newline\url{https://doi.org/10.1093/logcom/exl002}

\bibitem{Plonka1967}
J.~P\l{}onka.
\newblock On a method of construction of abstract algebras.
\newblock \emph{Fundam. Math.} \textbf{61}(2) (1967), 183--189.
\newline\url{https://doi.org/10.4064/fm-61-2-183-189}

\bibitem{Rees1940}
D.~Rees.
\newblock On semi-groups.
\newblock \emph{Proc. Camb. Philos. Soc.} \textbf{36} (1940), 387--400.
\newline\url{https://doi.org/10.1017/S0305004100017436}

\bibitem{Rosenthal1990}
K.~I. Rosenthal.
\newblock \emph{Quantales and Their Applications}.
\newblock Pitman Research Notes in Mathematics Series, vol.~234, Longman Scientific \& Technical, Harlow, 1990.

\bibitem{Sugimoto2024LocallyIntegral}
M.~Sugimoto.
\newblock Decompositions of locally integral involutive residuated structures.
\newblock In \emph{Abstracts of the Topology, Algebra, and Categories in Logic 2024 Conference}, Barcelona, 2024.

\bibitem{WardDilworth1939}
M.~Ward and R.~P. Dilworth.
\newblock Residuated lattices.
\newblock \emph{Trans. Am. Math. Soc.} \textbf{45} (1939), 335--354.
\newline\url{https://doi.org/10.1090/S0002-9947-1939-1501995-3}

\end{thebibliography}
\end{document}